\documentclass[11pt]{amsart}
\usepackage{amssymb}
\usepackage{amsfonts} 
\usepackage{amsmath}                           
\usepackage{amsthm}

\usepackage{hyperref} 
\usepackage[verbose]{geometry}
\usepackage[english]{babel}
\usepackage[latin1]{inputenc}
\usepackage{amsmath,amssymb,amsthm,epsfig}
\usepackage{eufrak}
\usepackage{color}
\usepackage{mathrsfs}
\usepackage{latexsym}

\newcommand{\R}{\Bbb{R}}

\theoremstyle{plain}
   
\newtheorem{teo}{Theorem}[section]

\newtheorem{defi}{Definition}[section]

\newtheorem{cor}{Corollary}[section]
\theoremstyle{definition}
\newtheorem{exe}{Example}[section]
\newtheorem{obs}{Remark}[section]
\begin{document}
\title{The Poincar\'e Problem for foliations on Compact Toric Orbifolds}

\author{Miguel Rodr\'iguez Pe\~na}
\address{A. M. Rodr\'iguez\\
DEMAT - UFSJ \\
Departamento de Matem\'atica e Estad\'istica\\
S\~ao Jo\~ao del-Rei MG, Brazil.}
\email{miguel.rodriguez.mat@ufsj.edu.br}

\date\today

\begin{abstract}
We give an optimal upper bound of the degree of quasi-smooth hypersurfaces which are invariant by a one-dimensional holomorphic foliation on a compact toric orbifold, i.e. on a complete simplicial toric variety. This bound depends only on the degree of the foliation and of the degrees of the toric homogeneous coordinates.
\end{abstract}

\maketitle


\section{Introduction}

Henri Poincar\'e studied in \cite{Poin} the problem to decide whether a holomorphic
foliation $\mathcal{F}$ on the complex projective plane $\mathbb{P}^2$ admits a rational first integral. Poincar\'e
observed that, in order to solve this problem, it is sufficient to find a bound for the degree of the generic $\mathcal{F}$-invariant curves.
Determining such a bound is known as the \textit{Poincar\'e problem}. Although it is well-known that such a bound does not exist in
general, under certain hypotheses, there are several works about Poincar\'e problem and its generalizations; see for instance 
\cite{BruMe}, \cite{Bru}, \cite{Can}, \cite{Alc}, \cite{EstKle}, \cite{Gal}, \cite{Pere} and \cite{So}.

Toric varieties form an important class of examples in algebraic geometry. Furthermore, its geometry is fully determined by the combinatorics of its associated fan, which often makes computations far more tractable. 
Recently, M. Corr\^ea presented a result of Darboux-Jouanolou-Ghys's type for one-dimensional holomorphic foliations on toric varieties; see for instance \cite{Darb} and \cite{Mau}. 
In this paper, we consider a one-dimensional holomorphic foliation $\mathcal{F}$ on a compact toric orbifold $\mathbb{P}_{\Delta}$, i.e., a compact toric variety with at most quotient singularities. 
A quasi-homogeneous hypersurface $V$ of $\mathbb{P}_{\Delta}$ is invariant by $\mathcal{F}$, if 
$V\setminus\mathrm{Sing}(\mathcal{F})\cup\mathrm{Sing}(V)$ is a union of leaves of $\mathcal{F}$.
Note that similarly to the complex projective space, one can consider the Poincar\'e problem for holomorphic one-dimensional foliations on a compact toric orbifold. It is possible since we can consider homogeneous coordinates in a toric variety and therefore we can define the notion of degree of a foliation and of an invariant quasi-smooth hypersurface. In order to provide a solution to Poincar\'e problem, i.e. to find a bound for the degree of a quasi-smooth hypersurface invariant by $\mathcal{F}$,
we give a normal form for quasi-homogeneous vector fields with a quasi-smooth hypersurface invariant by $\mathcal{F}$ on a compact toric orbifold; see for instance Theorems \ref{Mig1} and \ref{Mig4}. As a consequence of Theorems \ref{Mig1} and \ref{Mig4}, we will see that the solution to Poincar\'e problem on compact toric orbifolds is optimal. It is important to point out here that our main results improve and generalize the results obtained in \cite{MiMaFa} and \cite{CorMa} about the Poincar\'e problem in weighted projective spaces. Also we generalize
the results of \cite{MaCor}, where the authors studied the Poincar\'e problem for foliations on multiprojective complex spaces. 
Moreover, we give a bound for the degree of a quasi-smooth hypersurface invariant by $\mathcal{F}$ on rational normal scrolls, and for a compact toric orbifold surface with Weil divisor class group having torsion.
Finally, in Section \ref{App}, we will present families of examples of one-dimensional holomorphic foliations to show that our results are optimals and an example that shows that our hypotheses are necessary. 

Let $\mathbb{P}_{\Delta}$ be an $n$-dimensional compact toric orbifold, determined by a fan $\Delta$ in $N\simeq\mathbb{Z}^n$. 
As $\Delta(1)$ spans $N_{\mathbb{R}}=N\otimes_{\mathbb{Z}}\mathbb{R}\simeq\mathbb{R}^n$, we have 
$\mathbb{P}_{\Delta}$ is a geometric quotient $(\mathbb{C}^{n+r}-{\mathcal{Z}})/G$, where the group 
$G = Hom_{\mathbb{Z}}(\mathcal{A}_{n-1}(\mathbb{P}_{\Delta}), \mathbb{C}^*)$ acts on $\mathbb{C}^{n+r}$, 
$\mathcal{A}_{n-1}(\mathbb{P}_{\Delta})$ is the Weil divisor class group of $\mathbb{P}_{\Delta}$ and $\mathcal{Z}$ is an algebraic set of codimension at least two in $\mathbb{C}^{n+r}$; see \cite[Cox]{Cox}. Consider $\mathcal{T}\mathbb{P}_{\Delta}=
\mathcal{H}om(\Omega_{\mathbb{P}_{\Delta}}^1,\mathcal{O}_{\mathbb{P}_{\Delta}})$ the Zariski tangent sheaf of $\mathbb{P}_{\Delta}$.
Since $\mathbb{P}_{\Delta}$ is a complex orbifold then $\mathcal{T}\mathbb{P}_{\Delta} \simeq i_{\ast}
\mathcal{T}\mathbb{P}_{\Delta, \,\mathrm{reg}}$, where $i:\mathbb{P}_{\Delta, \,\mathrm{reg}} \rightarrow \mathbb{P}_{\Delta}$
is the inclusion of the regular part $\mathbb{P}_{\Delta, \,\mathrm{reg}}= \mathbb{P}_{\Delta} \setminus \mathrm{Sing}(\mathbb{P}_{\Delta})$ and $\mathcal{T}\mathbb{P}_{\Delta, \,\mathrm{reg}}$ is the tangent sheaf of $\mathbb{P}_{\Delta, \,\mathrm{reg}}$; see 
\cite[Appendix A.2]{Cox3}. A singular one-dimensional holomorphic foliation $\mathcal{F}$ on $\mathbb{P}_{\Delta}$ is a global section of 
$\mathcal{T}\mathbb{P}_{\Delta}\otimes\mathcal{L}$, where $\mathcal{L}$ is an invertible sheaf on $\mathbb{P}_{\Delta}$ 
and $\mathrm{codim}\,\mathrm{Sing}(\mathcal{F})\geq 2$.
Consider the homogeneous coordinate ring $\mathrm{S}=\mathbb{C}[z_1,\ldots,z_{n+r}]$ of $\mathbb{P}_{\Delta}$ 
and let $f\in\mathrm{S}$ be a quasi-homogeneous polynomial; see Subsection $\ref{coxhom}$.
We say that $V=\left\{f=0\right\}\subset\mathbb{P}_{\Delta}$ is a \textit{quasi-smooth hypersurface} if its tangent cone $\left\{f = 0\right\}$ 
on $\mathbb{C}^{n+r} \setminus \mathcal{Z}$ is smooth, and it is said to be \textit{strongly quasi-smooth hypersurface} if its tangent cone 
$\left\{f = 0\right\}$ on $\mathbb{C}^{n+r} \setminus \left\{0\right\}$ is smooth. Note that $V$ is quasi-smooth if and only if $V$ is a suborbifold of $\mathbb{P}_{\Delta}$, and both definitions coincide when the rank of $\mathcal{A}_{n-1}(\mathbb{P}_{\Delta})$ is one;
see for instance \cite{Cox1} and \cite{Cox}. Let $X$ be a quasi-homogeneous vector field which defines $\mathcal{F}$ in homogeneous coordinates. Then $V=\left\{f=0\right\}\subset\mathbb{P}_{\Delta}$ is invariant by $\mathcal{F}$ if $X(f) = g\cdot f$ for some
quasi-homogeneous polynomial $g$. 

The Weil divisor class group $\mathcal{A}_{n-1}(\mathbb{P}_{\Delta})$ is a finitely generated abelian group of rank $r$. By the fundamental theorem of finitely generated abelian groups, we have
\begin{equation}\label{eq00}
\begin{array}{ccccccccc}
\mathcal{A}_{n-1}(\mathbb{P}_{\Delta})\simeq\mathbb{Z}^{\,r} \oplus \mathbb{Z}_{(p_1)^{\lambda_1}} \oplus \cdots \oplus \mathbb{Z}_{(p_m)^{\lambda_m}},
\end{array}
\end{equation}
where $p_i$ are primes, not necessarily distinct, and $\lambda_i$ are positive integers. The 
direct sum is unique except for possible rearrangement of the factors. Suppose that the homogeneous coordinate ring $\mathrm{S}$ of 
$\mathbb{P}_{\Delta}$ has the following property: there is a positive integer number 
$1 \leq k \leq r$ such that
\begin{equation}\label{eq0}
\begin{array}{ccccccccc}
(\deg \mathrm{P})_k \geq 0\,\,\textup{for all}\,\,\mathrm{P} \in \mathrm{S},
\end{array}
\end{equation}
where $(\deg \mathrm{P})_k$ denotes the $k$-th integer coordinate of $\deg\mathrm{P}\in\mathcal{A}_{n-1}(\mathbb{P}_{\Delta})$,
i.e. the $k$-th component of the multidegree $\deg\mathrm{P}$ in \eqref{eq00}. First, we consider the question of bounding the degree of strongly quasi-smooth hypersurfaces which are invariant by a one-dimensional holomorphic foliation on a compact toric orbifold:

\begin{teo} Let $\mathbb{P}_{\Delta}$ be a complete simplicial toric variety of dimension $n$, with homogeneous coordinates 
$z_1,\ldots,z_{n+r}$. Let $\mathcal{F}$ be a one-dimensional holomorphic foliation on $\mathbb{P}_{\Delta}$ and let $X$ be a quasi-homogeneous vector field which defines $\mathcal{F}$ in homogeneous coordinates. Let $V=\left\{f=0\right\}\subset\mathbb{P}_{\Delta}$ be a strongly quasi-smooth hypersurface invariant by 
$\mathcal{F}$. Then 
$$\deg(V)_k \leq \deg(\mathcal{F})_k + \max_{1\leq i<j\leq n+r} \big\{\deg(z_i)_k + \deg(z_j)_k\big\},$$
for each $1\leq k \leq r$ as in the hypothesis $(\ref{eq0})$.
\end{teo}

There is a situation in which the above theorem is valid for a quasi-smooth hypersurface 
$V=\left\{f=0\right\}\subset\mathbb{P}_{\Delta}$ with
$$\left\{0\right\} \subsetneq \mathrm{Sing}(V) \subset \mathcal{Z}.$$
Here we are going to consider a variant of the previous theorem for quasi-smooth hypersurfaces on a compact toric orbifold:

\begin{teo} Let $\mathbb{P}_{\Delta}$ be a complete simplicial toric variety of dimension $n$, with homogeneous coordinates 
$z_1,\ldots,z_{n+r}$. Let $\mathcal{F}$ be a one-dimensional holomorphic foliation on
$\mathbb{P}_{\Delta}$ and let $X$ be a quasi-homogeneous vector field which defines $\mathcal{F}$ in homogeneous coordinates. 
Let $V=\left\{f=0\right\}\subset\mathbb{P}_{\Delta}$ be a quasi-smooth hypersurface. 
Suppose there are integer numbers $1\leq i_1<\cdots<i_k\leq n+r$ such that 
\begin{enumerate}
 	\item [(i)] there is a regular subsequence 
$\left\{\frac{\partial f}{\partial z_{i_1}},\dots,\frac{\partial f}{\partial z_{i_k}}\right\}\subset
\left\{\frac{\partial f}{\partial z_{1}},\dots,\frac{\partial f}{\partial z_{n+r}}\right\}$ and
  \item [(ii)] there is a radial vector field $R_{i_1,\ldots,i_k}=\sum_{j=1}^{k} a_{i_j}z_{i_j}\frac{\partial}{\partial z_{i_j}}$ such that
$i_{R_{i_1,\ldots,i_k}}(df)=\theta(\alpha)\cdot f$, where $\alpha\in\mathcal{A}_{n-1}(\mathbb{P}_{\Delta})$ is the degree of $f$ and 
$\theta(\alpha)$ is a constant; see Subsection $\ref{Qsh}$.
\end{enumerate}
Moreover assume that, in homogeneous coordinates, $X=X_1+X_2$, where $X_1=\sum_{j=1}^{k} P_{i_j}\frac{\partial}{\partial z_{i_j}}$
leaves $V$ invariant and such that $X_1\notin \mathrm{Lie}(G)$. Then 
$$\deg(V)_{\ell} \leq \deg(\mathcal{F})_{\ell} + \max_{1\leq j_1<j_2\leq k} \big\{\deg(z_{i_{j_1}})_{\ell}
+ \deg(z_{i_{j_2}})_{\ell}\big\},$$
for each $1\leq \ell \leq r$ as in the hypothesis $(\ref{eq0})$. Here, $\mathrm{Lie}(G)$ is defined in Subsection $\ref{Qsh}$.
\end{teo}

Note that the property $($i$)$ is equivalent to
$$\mathrm{codim}\left( \left\{\frac{\partial f}{\partial z_{i_1}}=\cdots=\frac{\partial f}{\partial z_{i_s}}=0\right\}\right)
=s,\,s\leq k;$$
for more details see \cite{Grif}.
  
\bigskip



\section{Generalities on toric varieties}

In this section we recall some basic definitions and results about simplicial complete toric
varieties. For more details about toric varieties see \cite{Bras}, \cite{Cox2}, \cite{Cox}, \cite{Fu} and \cite{Oda}.

Let $N$ be a free $\mathbb{Z}$-module of rank $n$ and $M =
\mathrm{Hom}(N, \mathbb{Z})$ be its dual. A subset $\sigma\subset
N_{\mathbb{R}}=N\otimes_{\mathbb{Z}} \mathbb{R}\simeq \mathbb{R}^n$ is called a
strongly convex rational polyhedral cone if there exists a finite
number of elements $v_1,\dots,v_{k}$ in the lattice $N$ such that
$$
\sigma=\{a_1 v_1+\cdots+a_k v_{k} : a_i\in \mathbb{R}, a_i\geq 0\},
$$
and $\sigma$ does not contain any straight line going through the
origin, i.e. $\sigma\cap(-\sigma)=\left\{0\right\}$.
A cone $\sigma$ is called simplicial if its
generators can be chosen to be linearly independent over $\R$. The
dimension of a cone $\sigma$ is, by definition, the dimension of a
minimal subspace of $\mathbb{R}^n$ containing $\sigma$.

Set $M_{\mathbb{R}}=M\otimes_{\mathbb{Z}} \mathbb{R}$ and 
$\left\langle\,\,,\,\right\rangle:M_{\mathbb{R}}\times N_{\mathbb{R}}\longrightarrow\mathbb{R}$ the
dual pairing. To each cone $\sigma$ we associate the dual cone $\check{\sigma}$
$$
\check{\sigma}=\{m\in M_{\mathbb{R}} : \langle m,
v \rangle \geq 0\,\,\, \forall v\in \sigma\},
$$
which is a rational polyhedral cone in $M_{\mathbb{R}}$. It
follows from Gordan's Lemma that $\check{\sigma}\cap M$ is also a finitely generated
semigroup. A subset $\tau$ of $\sigma$ is called a face and is denoted $\tau\prec\sigma$, if 
$$\tau=\sigma\cap\left\{m\right\}^{\bot}=\left\{v\in\sigma : \langle m, v \rangle = 0\right\},$$
for some $m\in\check{\sigma}$. A cone is a face of itself, other faces are called proper faces.

\begin{defi}
A non-empty collection $\Delta=\{\sigma_1,\dots,\sigma_s\}$ of
strongly convex rational polyhedral cones in
$N_\mathbb{R}\simeq \mathbb{R}^n$
is called a fan if it satisfies:
\begin{itemize}
  \item [(i)] if $\sigma \in \Delta$ and $\tau\prec\sigma$, then $\tau\in \Delta,$
  \item [(ii)] if $\sigma_i,\sigma_j\in \Delta$, then $\sigma_i\cap\sigma_j\prec
  \sigma_i$ and $\sigma_i\cap\sigma_j\prec\sigma_j.$
\end{itemize}
\end{defi}

The fan $\Delta$ is called complete if $N_\mathbb{R}=\sigma_1\cup\cdots\cup\sigma_s$.
The dimension of a fan is the maximal dimension of its cones. An
$n$-dimensional complete fan is called simplicial if all its
$n$-dimensional cones are simplicial. An affine $n$-dimensional toric variety corresponding to
$\sigma$ is the variety 
$$
\mathcal{U}_{\sigma}= \mathrm{Spec} \, \mathbb{C}[\check{\sigma}\cap M].
$$
If a cone $\tau$ is a face of $\sigma$ then $\check{\tau}\cap M$ is
a subsemigroup of $\check{\sigma}\cap M$, hence $\mathcal{U}_{\tau}$
is embedded into $\mathcal{U}_{\sigma}$ as an open subset. The
affine varieties corresponding to all cones of the fan $\Delta$ are
glued together according to this rule into the toric variety
$\mathbb{P}_{\Delta}$ associated with $\Delta$. Is possible to show
that a toric variety $\mathbb{P}_{\Delta}$ contain  a complex torus
$\mathbb{T}^n=(\mathbb{C}^*)^n$ as a Zariski open subset such that
the action of $\mathbb{T}^n$ on itself extends to an action of
$\mathbb{T}^n$ on $\mathbb{P}_{\Delta}$. A toric variety $\mathbb{P}_{\Delta}$ determined by
a complete simplicial fan $\Delta$ is a compact complex orbifold, i.e. a compact complex 
variety with at most quotient singularities. Note that $\mathbb{T}^n$, $\mathbb{C}^n$ and $\mathbb{P}^n$ are toric varieties. 

\subsection{The homogeneous coordinate ring} \label{coxhom}

Let $\mathbb{P}_{\Delta}$ be the toric variety determined by a fan
$\Delta$ in $N\simeq\mathbb{Z}^n$. The one-dimensional cones of $\Delta$ form the set
$\Delta(1)$, and given $\rho\in\Delta(1)$, we set $n_\rho$ the 
unique generator of $\rho\,\cap N$. If $\sigma$ is any cone in $\Delta$, then 
$\sigma(1)=\{\rho\in \Delta(1)\,:\,\rho \prec \sigma\}$ is the set of
one-dimensional faces of $\sigma$. We will assume that $\Delta(1)$
spans $N_{\mathbb{R}}=N\otimes_{\mathbb{Z}}\mathbb{R}\simeq\mathbb{R}^n$.

Each $\rho \in \Delta(1)$ corresponds to an irreducible 
$\mathbb{T}$-invariant Weil divisor $D_{\rho}$ in
$\mathbb{P}_{\Delta}$, where
$\mathbb{T}=N\otimes_{\mathbb{Z}}\mathbb{C}^{*} \simeq Hom_{\mathbb{Z}}(M,\mathbb{C}^*)$ is the torus acting
on $\mathbb{P}_{\Delta}$. The $\mathbb{T}$ -invariant Weil divisors
on $\mathbb{P}_{\Delta}$ form a free abelian group of rank $|{\Delta}(1)|$, that will be denoted $\mathbb{Z}^{{\Delta}(1)}$. 
Thus an element $D\in \mathbb{Z}^{{\Delta}(1)}$ is a sum $D=\sum_{\rho}a_{\rho}D_{\rho}$.
The $\mathbb{T}$-invariant Cartier divisors form a subgroup $\mathrm{Div}_{\mathbb{T}}(\mathbb{P}_{\Delta})\subset\mathbb{Z}^{\Delta(1)} $.

Each $m\in M$ gives a character $\chi^m:\mathbb{T}\rightarrow\mathbb{C}^*$, and hence $\chi^m$ is a
rational function on $\mathbb{P}_{\Delta}$. As is well-known,
$\chi^m$ gives the $\mathbb{T}$-invariant Cartier divisor $\mathrm{div}(\chi^m) = -\sum_{\rho}\langle m, n_{\rho}\rangle D_{\rho}$. 
We will consider the map
$$
\begin{array}{ccc}
  M & \longrightarrow& \mathbb{Z}^{\Delta(1)} \\
 m & \longmapsto& D_{m}=\sum_{\rho}\langle m, n_{\rho} \rangle D_{\rho}.
\end{array}
$$
\\
This map is injective since $\Delta(1)$ spans $N_{\mathbb{R}}$. By
\cite{Fu}, we have a commutative diagram

\begin{equation}\label{eq1}
\begin{array}{ccccccccc}
  0 & \rightarrow & M & \rightarrow &\mathrm{Div}_{\mathbb{T}}(\mathbb{P}_{\Delta})
    & \rightarrow &\mathrm{Pic}(\mathbb{P}_{\Delta})&  \rightarrow & 0 \\
&&\shortparallel& &\downarrow& &\downarrow& &
\\
  0 & \rightarrow & M & \rightarrow & \mathbb{Z}^{\Delta(1)}
   & \rightarrow & \mathcal{A}_{n-1}(\mathbb{P}_{\Delta}) & \rightarrow &
  0
\end{array}
\end{equation}
For each $\rho\in \Delta(1)$, introduce a variable $z_{\rho}$,
and consider the polynomial ring
$$
\mathrm{S}=\mathbb{C}[z_{\rho}]=\mathbb{C}[z_{\rho} : \rho\in\Delta(1)].
$$
Note that a monomial $\prod_{\rho}z_{\rho}^{a_{\rho}}$
determines a divisor $D=\sum_{\rho}a_{\rho}D_{\rho}$ and to emphasize
this relationship, we will write the monomial as $z^{D}$. We will
grade $\mathrm{S}$ as follows: the degree of a monomial $z^{D}\in \mathrm{S}$ is
$\deg(z^{D}) = [D]\in \mathcal{A}_{n-1}(\mathbb{P}_{\Delta})$. 
Using the exact sequence \eqref{eq1}, it follows that two monomials 
$\prod_{\rho}z_{\rho}^{a_{\rho}}$ and $\prod_{\rho}z_{\rho}^{b_{\rho}}$ in $\mathrm{S}$ have the same degree if
and only if there is some $m\in M$ such that $a_{\rho}=\langle m, n_{\rho}\rangle+b_{\rho}$ for every $\rho$. Then
$$
\mathrm{S}=\bigoplus_{\alpha \in
A_{n-1}(\mathbb{P}_{\Delta})}\mathrm{S}_{\alpha},
$$
where
$\mathrm{S}_{\alpha}=\bigoplus_{\deg(z^{D})=\alpha}\mathbb{C} \cdot z^{D}$. Note also that $\mathrm{S}_{\alpha}\cdot \mathrm{S}_{\beta}
=\mathrm{S}_{\alpha+\beta}$. 
The polynomial ring $\mathrm{S}$ is called homogeneous coordinate ring of the toric variety $\mathbb{P}_{\Delta}$.

Denote by $\mathcal{O}_{\mathbb{P}_{\Delta}}$ the structure sheaf of $\mathbb{P}_{\Delta}$. Let $\mathcal{O}_{\mathbb{P}_{\Delta}}(D)$ be the coherent sheaf on $\mathbb{P}_{\Delta}$ determined by a Weil divisor $D$. If $\alpha = [D] \in \mathcal{A}_{n-1}(\mathbb{P}_{\Delta})$, then it follows from \cite{Cox} that
$$
\mathrm{S}_{\alpha}\simeq \mathrm{H}^0(\mathbb{P}_{\Delta},\mathcal{O}_{\mathbb{P}_{\Delta}}(D)),
$$
moreover, if $\alpha=[D_1]$ and $\beta=[D_2]$, there is a commutative diagram
$$
\begin{array}{ccc}
  \mathrm{S}_{\alpha}\otimes \mathrm{S}_{\beta} & \longrightarrow & \mathrm{S}_{\alpha+\beta} \\
  \downarrow &  & \downarrow \\
   \mathrm{H}^0(\mathbb{P}_{\Delta},\mathcal{O}_{\mathbb{P}_{\Delta}}(D_1))\otimes \mathrm{H}^0(\mathbb{P}_{\Delta},\mathcal{O}_{\mathbb{P}_{\Delta}}(D_2))& \longrightarrow & \mathrm{H}^0(\mathbb{P}_{\Delta},\mathcal{O}_{\mathbb{P}_{\Delta}}(D_{1}+D_{2}))
\end{array}
$$
where the top arrow is the polynomial multiplication. If
$\mathbb{P}_{\Delta}$ is a complete toric variety, then it follows from \cite{Cox} that 
\begin{itemize}
  \item [(i)] $\mathrm{S}_{\alpha}$ is finite dimensional for every $\alpha$, and in
particular, $\mathrm{S}_0 = \mathbb{C}$.
  \item [(ii)] If $\alpha= [D]$ for an effective divisor
$D=\sum_{\rho}a_{\rho}D_{\rho}$, then
$\mathrm{dim}_{\mathbb{C}}(\mathrm{S}_{\alpha})=\#(P_{D}\cap M)$, where
$P_{D}=\{m\in M_{\mathbb{R}} : \langle m,
n_{\rho}\rangle\geq -a_{\rho}$ for all $\rho\}$.
\end{itemize}

\subsection{The toric homogeneous coordinates}

Given a toric variety $\mathbb{P}_{\Delta}$, the Weil divisor class group $\mathcal{A}_{n-1}(\mathbb{P}_{\Delta})$ is a finitely generated
abelian group of rank $r=k-n$, where $k=|\Delta(1)|$. If we apply $Hom_{\mathbb{Z}}(-,\mathbb{C}^*)$ to the bottom exact sequence of 
\eqref{eq1}, then we get the exact sequence 
$$1 \longrightarrow G \longrightarrow (\mathbb{C}^*)^{\Delta(1)}\longrightarrow \mathbb{T} \longrightarrow 1\,,$$
where $G = Hom_{\mathbb{Z}}(\mathcal{A}_{n-1}(\mathbb{P}_{\Delta}), \mathbb{C}^*)$. Since $(\mathbb{C}^*)^{\Delta(1)}$ acts naturally on 
$\mathbb{C}^{\Delta(1)}$, the subgroup $G\subset (\mathbb{C}^*)^{\Delta(1)}$ 
acts on $\mathbb{C}^{\Delta(1)}$ by 
$$g\cdot t=\Big(g\big([D_{\rho}]\big) \, t_{\rho}\Big)\,,$$
for $g:\mathcal{A}_{n-1}(\mathbb{P}_{\Delta}) \rightarrow \mathbb{C}^*$ in $G$, and $t=(t_{\rho})\in\mathbb{C}^{\Delta(1)}$. The explicit equations for $G$ as a subgroup of the torus $(\mathbb{C^*})^{\Delta(1)}$ is given by 
$$G=\left\{(t_{\rho})\in(\mathbb{C}^*)^{\Delta(1)}\, 
\middle| \, \prod_{\rho} t_{\rho}^{\left\langle\, m_{i},\,n_{\rho}\right\rangle}=1, 
\,\, 1\leq i\leq n\right\},$$
where $m_1,\dots,m_n$  is a basis of $M$; see \cite{Cox2}.

For each cone $\sigma \in \Delta$, define the monomial
$$z^{\widehat{\sigma}}=\prod_{\rho\notin \sigma(1)}z_{\rho}\,,$$
which is the product of the variables corresponding to rays not in $\sigma$. Then define 
$$\mathcal{Z}=V(\{z^{\widehat{\sigma}}\,:\,\sigma \in \Delta\})\subset \mathbb{C}^{\Delta(1)}.$$ 
We have that $\mathcal{Z} \subset \mathbb{C}^{\Delta(1)}$ has codimension at least two, and 
$\mathcal{Z}=\left\{0\right\}$ when $r=1$; see \cite{Cox}. 

\begin{teo}\cite{Cox}\label{Cox}
Let $\mathbb{P}_{\Delta}$ be a $n$-dimensional toric variety such that $\Delta(1)$ spans $N_{\mathbb{R}}$. Then

\begin{enumerate}

  \item [(i)] The set $\mathbb{C}^{\Delta(1)}-{\mathcal{Z}}$ is invariant under the action of the group $G$.

  \item [(ii)] $\mathbb{P}_{\Delta}$ is naturally isomorphic to the categorical quotient $(\mathbb{C}^{\Delta(1)}-{\mathcal{Z}})/G$.

  \item [(iii)] $\mathbb{P}_{\Delta}$ is the geometric quotient $(\mathbb{C}^{\Delta(1)}-{\mathcal{Z}})/G$ if and only if
	$\mathbb{P}_{\Delta}$ is an orbifold.

\end{enumerate}

\end{teo}

\subsection{Quasi-smooth hypersurfaces}\label{Qsh}

Let $\mathbb{P}_{\Delta}$ be a complex orbifold. An element $\alpha\in\mathcal{A}_{n-1}(\mathbb{P}_{\Delta})$
gives the character $\chi^{\alpha}:G \rightarrow \mathbb{C}^*$ that sends $g\in G=Hom_{\mathbb{Z}}(\mathcal{A}_{n-1}(\mathbb{P}_{\Delta}), \mathbb{C}^*)$ to $g(\alpha)\in\mathbb{C}^*$.
The action of $G$ on $\mathbb{C}^{\Delta(1)}$ induces
an action on $\mathrm{S}$ with the property that given $f\in \mathrm{S}$, we have
$$
f\in \mathrm{S}_{\alpha }\Leftrightarrow f(g\cdot z)= \chi^{\alpha}(g)f(z), \
\ \forall \ g\in G, \ z\in \mathbb{C}^{\Delta(1)}.
$$
We say that $f\in \mathrm{S}_{\alpha }$ is quasi-homogeneous of
degree $\alpha$. It follows that the equation $V=\left\{f=0\right\}$ is
well-defined in $\mathbb{P}_{\Delta}$ and it defines a hypersurface. We say that $V=\left\{f=0\right\}$ is a quasi-homogeneous hypersurface
of degree $\alpha$. We say that $V=\left\{f=0\right\}$ is \textit{quasi-smooth} if its tangent cone $\left\{f=0\right\}$ on $\mathbb{C}^{\Delta(1)}-{\mathcal{Z}}$ is smooth. We have the following theorem.

\begin{teo}\cite{Cox1}\label{Cox1}
Let $V=\left\{f=0\right\}\subset \mathbb{P}_{\Delta}$ be a quasi-homogeneous hypersurface.
Then $V$ is quasi-smooth if and only if $V$ is a suborbifold of $\mathbb{P}_{\Delta}$.
\end{teo}

Suppose there is a complex number $a_{\rho}$ for each $\rho\in\Delta(1)$ with the property that $\sum_{\rho}a_{\rho} n_{\rho}=0$
in $N_{\mathbb{C}}$. Then, for any class $\alpha \in \mathcal{A}_{n-1}(\mathbb{P}_{\Delta})$, there is a constant 
$\theta(\alpha)$ with the property that for any quasi-homogeneous polynomial $f \in \mathrm{S}$ of degree $\alpha$, we have
\begin{equation}\label{eq2}   
\begin{array}{ccccccccc}
i_{R}(df)=\theta(\alpha)\cdot f,
\end{array}
\end{equation}
where $R=\sum_{\rho}a_{\rho}z_{\rho}\frac{\partial}{\partial z_{\rho}}$. The identity \eqref{eq2} is called the Euler formula determined by 
$\left\{a_{\rho}\right\}$. Moreover, considering the $r=k-n$ linearly independent over $\mathbb{Z}$ relations among the $n_{\rho}$, we have $r$ vector fields $R_i$ tangent to the orbits of $G$ and $\mathrm{Lie}(G)=\langle R_1,\dots,R_r\rangle$; for more details see 
\cite{Cox1}. We will call these vector fields $R_i, i=1,\ldots,r$, the radial vector fields on $\mathbb{P}_{\Delta}$. 

We shall consider the following subfield of $\mathbb{C}(z_{\rho})=\mathrm{Frac}(\mathbb{C}[z_{\rho}])$ given by

$$
\widetilde{K}(\mathbb{P}_{\Delta})=\left\{\frac{P}{Q}\in
\mathbb{C}(z_{\rho}) : P\in \mathrm{S}_{\alpha}, Q\in \mathrm{S}_{\beta}\right\}.
$$
Thus, the field of rational functions on $\mathbb{P}_{\Delta}$,
denoted by $K(\mathbb{P}_{\Delta})$, is the subfield of
$\widetilde{K}(\mathbb{P}_{\Delta})$ such that $\deg(P)=\deg(Q)$.  
It follows that the polynomials $P,Q \in \mathrm{S}_{\alpha}$ define a rational function 
$\frac{P}{Q}:\mathbb{P}_{\Delta} \dashrightarrow \mathbb{P}^1$.

\subsection{Examples}

\bigskip

Let $\mathbb{P}_{\Delta}$ be a $n$-dimensional toric variety where $\Delta(1)$
spans $N_{\mathbb{R}}$. We know that $|\Delta(1)|=n+r$, where $r$ is the 
rank of the finitely generated abelian group $\mathcal{A}_{n-1}(\mathbb{P}_{\Delta})$. 
We will denote 
$\Delta(1)=\left\{\rho_1,\dots,\rho_{n+r}\right\}$, $\mathrm{S}=\mathbb{C}\left[z_1,\dots,z_{n+r}\right]$, and $D_i=D_{\rho_i}$ for all $i=1,\dots,n+r$. \\

\begin{enumerate}
	\item \label{exe1} \textbf{Weighted projective spaces.} \cite{Cox2} Let $\omega_0,\dots,\omega_n$ be positive integers with 
$gcd(\omega_0,\dots,\omega_n)=1$.
Set $\omega=(\omega_{0},\dots,\omega_{n})$. Consider the lattice $N=\mathbb{Z}^{n+1} / \mathbb{Z}\cdot\omega$. The dual lattice is 
$$M=\left\{(a_0,\dots,a_n)\in\mathbb{Z}^{n+1}\,|\,a_0\omega_0+\cdots+a_n\omega_n=0\right\}.$$ 
Denote by $e_0,\dots,e_n$ the standard basis of $\mathbb{Z}^{n+1}$. We have the exact sequence
$$0\longrightarrow M\stackrel{\alpha}{\longrightarrow}\mathbb{Z}^{n+1}\stackrel{\beta}{\longrightarrow}\mathbb{Z}\longrightarrow0,$$
where $\alpha(m)=\left(\left\langle m,\bar{e}_0\right\rangle,\dots,\left\langle m,\bar{e}_n\right\rangle\right)$ and 
$\beta(a_0,\dots,a_n)=a_0\omega_0+\cdots+a_n\omega_n$. Let $\Delta$ be the fan made up of the cones generated by all the proper subsets of 
$\{\bar{e}_0,\dots,\bar{e}_n\}$. Then, $G$ is given by $G=\left\{(t^{\omega_0},\dots,t^{\omega_n})\,|\,t\in \mathbb{C}^* \right\}\simeq\mathbb{C}^{*}$ and its action on $\mathbb{C}^{n+1}$ is given by 
$$t\cdot(z_0,\dots,z_n)=(t^{\omega_0} z_0,\dots, t^{\omega_n} z_n).$$
Since $\Delta$ is simplicial and has $n+1$ rays, we have $\mathcal{Z}=\left\{0\right\}$ and
$$\mathbb{P}(\omega) := \mathbb{P}_\Delta = \left(\mathbb{C}^{n+1} - \{0\}\right) / \mathbb{C}^*,$$ 
is the usual representation of weighted projective spaces as a quotient. If $\omega_0=\cdots=\omega_n=1$, then 
$\mathbb{P}(\omega)=\mathbb{P}^{n}$ and when $\omega_0,\dots,\omega_n$ are pairwise coprime, we have 
$$\mathrm{Sing}(\mathbb{P}(\omega))=\left\{\bar{e}_i\,|\,\omega_i>1\right\}.$$
Moreover, we have $\mathcal{A}_{n-1}(\mathbb{P}(\omega))\simeq\mathbb{Z}$ and $\deg(z_i)=\omega_i$ for all $0\leq i\leq n+1$.
Consequently the homogeneous coordinate ring of $\mathbb{P}(\omega)$ is given by
$\mathrm{S}=\oplus_{\alpha\geq 0}\mathrm{S}_{\alpha}$, where 
$$\mathrm{S}_{\alpha}=\bigoplus_{p_0\omega_0+\cdots+p_n\omega_n=\alpha} \mathbb{C}\cdot {z_0}^{p_0}\dots\,{z_n}^{p_n}.$$\\
 \item \label{exe2} \textbf{Multiprojective spaces.} \cite{Cox2} Let $e_{1,1},\dots,e_{1,n}$ be a basis of $N_1\simeq\mathbb{Z}^n$, and set 
$e_{1,0}=-e_{1,1}-\cdots-e_{1,n}$. If $\Delta_1$ is the fan in $(N_1)_{\mathbb{R}}$ made up of the cones generated by all the proper subsets of $\left\{e_{1,0},\dots,e_{1,n}\right\}$, then $\mathbb{P}_{\Delta_1}=\mathbb{P}^n$. Similarly, if $e_{2,1},\dots,e_{2,m}$ is a basis of 
$N_2\simeq\mathbb{Z}^m$ and $e_{2,0}=-e_{2,1}-\cdots-e_{2,m}$, then we have 
$\mathbb{P}_{\Delta_2}=\mathbb{P}^m$. If $N=N_1 \oplus N_2$, then $\Delta=\Delta_1\times\Delta_2$ is a fan in $N_{\mathbb{R}}$ and
$\mathbb{P}_{\Delta}=\mathbb{P}_{\Delta_1}\times\mathbb{P}_{\Delta_2}=\mathbb{P}^n\times\mathbb{P}^m$. Set $a=(a_1,\dots,a_n)\in\mathbb{Z}^n$ and $b=(b_1,\dots,b_m)\in\mathbb{Z}^m$, we have the exact sequence
$$0\longrightarrow \mathbb{Z}^n \oplus \mathbb{Z}^m \stackrel{\alpha}{\longrightarrow}
\mathbb{Z}^{n+1}\oplus \mathbb{Z}^{m+1}\stackrel{\beta}{\longrightarrow}\mathbb{Z}\oplus\mathbb{Z}\longrightarrow0,$$
where $\alpha(a,b)=(-a_1-\cdots-a_n,a,-b_1-\cdots-b_m,b)$ and $\beta(a_0,a,b_0,b)=(a_0+\cdots+a_n,b_0+\cdots+b_m)$. Then, $G$ is given by
$G=\left\{(\mu,\dots,\mu,\lambda,\dots,\lambda)\in(\mathbb{C}^*)^{n+1}\times(\mathbb{C}^*)^{m+1}\right\}\simeq\mathbb{C}^{*}\times\mathbb{C}^{*}$ and its action on $\mathbb{C}^{n+1}\times\mathbb{C}^{m+1}$ is given by 
$$(\mu,\lambda)\cdot(z_1,z_2)=(\mu\cdot z_1,\lambda\cdot z_2)=(\mu z_{1,0},\dots,\mu z_{1,n},\lambda z_{2,0},\dots,\lambda z_{2,m}).$$
It is possible to show that  $\mathcal{Z}=\left(\left\{0\right\}\times\mathbb{C}^{m+1}\right) \cup \left(\mathbb{C}^{n+1}\times\left\{0\right\}\right)$. So, we have
$$\mathbb{P}^n\times\mathbb{P}^m=\left(\mathbb{C}^{n+1}\times\mathbb{C}^{m+1}-\mathcal{Z}\right)\,/\,(\mathbb{C}^{*})^2,$$
is the usual representation of $\mathbb{P}^n\times\mathbb{P}^m$ as a quotient space. 

Moreover, we have $\mathcal{A}_{n+m-1}(\mathbb{P}^n\times\mathbb{P}^m)\simeq\mathbb{Z}^2$ and
$\deg(z_{1,i})=(1,0)$ for all $0\leq i\leq n$, $\deg(z_{2,j})=(0,1)$ for all $0\leq j\leq m$.
Consequently the homogeneous coordinate ring of $\mathbb{P}^n\times\mathbb{P}^m$ is given by 
$\mathrm{S}=\oplus_{\alpha,\,\beta\geq 0}\mathrm{S}_{(\alpha,\beta)}$, where 
$$\mathrm{S}_{(\alpha,\beta)}=\bigoplus_{p_0+\cdots+p_n=\alpha\,;\,q_0+\cdots+q_m=\beta} \mathbb{C}\cdot {z_{1,0}}^{p_0}\dots\,{z_{1,n}}^{p_n} 
{z_{2,0}}^{q_0}\dots\,{z_{2,m}}^{q_m},$$
is the ring of bihomogeneous polynomials of bidegree $(\alpha,\beta)$.\\
 \item \label{exe3} \textbf{Rational normal scrolls.} Let $N\simeq\mathbb{Z}^2$ and $1\leq a\leq b$ integers. Consider the polygon 
$$P_{a,b}=Conv(0,ae_1,e_2,be_1+e_2)\subset M_{\mathbb{R}}\simeq\mathbb{R}^2.$$
The polygon $P_{a,b}$ has $a+b+2$ lattice points. A rational normal scroll $\mathbb{F}(a,b)$ is the toric surface associated to the 
normal fan of $P_{a,b}$. A rational normal scroll is a smooth projective surface because $P_{a,b}$ is full dimensional smooth lattice polytope.

Consider the map $\varphi:\mathbb{C}^*\times\mathbb{C}^*\rightarrow\mathbb{P}^{a+b+2}$
$$\varphi(s,t)=(1:s:s^2:\dots:s^a:t:st:s^2t:\dots:s^bt).$$
Then $P_{a,b}$ is the Zariski closure of the image of $\varphi$. Rewriting the map as
$\tilde{\varphi}:\mathbb{C}\times\mathbb{P}^1\rightarrow\mathbb{P}^{a+b+2}$
$$\tilde{\varphi}(s,(t_1:t_2))=(t_1:st_1:s^2t_1:\dots:s^at_1:t_2:st_2:s^2t_2:\dots:s^bt_2),$$
we have $s\mapsto\tilde{\varphi}(s,(1:0))$ and $s\mapsto\tilde{\varphi}(s,(0:1))$ are the rational normal curves $\mathcal{C}_a\subset\mathbb{P}^{a+1}\subset\mathbb{P}^{a+b+2}$ and $\mathcal{C}_b\subset\mathbb{P}^{b+1}\subset\mathbb{P}^{a+b+2}$. 

The rational normal scrolls are Hirzebruch surfaces because the normal fan of $P_{a,b}$ defines a Hirzebruch surface $\mathcal{H}_{b-a}$,
so $\mathbb{F}(a,b)\simeq\mathcal{H}_{b-a}$. Analogously, for a $n$-dimensional rational normal scroll, consider  $1\leq a_1 \leq a_2 \leq \dots \leq a_n$ integers, we have a full dimensional smooth lattice polytope $P_{a_1,\dots,a_n}\subset\mathbb{R}^n$ having $2n$ lattice points as vertices. A rational normal scroll $\mathbb{F}(a_1,\dots,a_n)$ is the smooth projective toric variety associated to the normal fan of $P_{a_1,\dots,a_n}$. It is possible to show that $\mathbb{F}(a_1,\dots,a_n)\simeq\mathbb{P}(\mathcal{O}_{\mathbb{P}^1}(a_1)\oplus\dots\oplus\mathcal{O}_{\mathbb{P}^1}(a_n))$. For more details see \cite{Cox2}.

In general, let $a_1,\ldots,a_n$ be integers. Consider the $(\mathbb{C}^*)^2$ action on $\mathbb{C}^2\times\mathbb{C}^n$ given as follows 
$$(\lambda,\mu)(z_{1,1},z_{1,2},z_{2,1},\dots,z_{2,n})=(\lambda z_{1,1}, \lambda z_{1,2}, \mu\lambda^{-a_1} z_{2,1},
\dots,\mu\lambda^{-a_n} z_{2,n}).$$
Then 
$$\mathbb{F}(a_1,\dots,a_n)=(\mathbb{C}^2\times\mathbb{C}^n-\mathcal{Z})\,/\,(\mathbb{C}^*)^2,$$
where $\mathcal{Z}=\left(\left\{0\right\}\times\mathbb{C}^{n}\right) \cup \left(\mathbb{C}^{2}\times\left\{0\right\}\right)$.

Moreover, we have $\mathcal{A}_{n-1}(\mathbb{F}(a_1,\dots,a_n))\simeq\mathbb{Z}^2$ and the homogeneous coordinate ring associated to 
$\mathbb{F}(a_1,\dots,a_n)$ is given by $\mathrm{S}=\oplus_{\alpha\in\mathbb{Z},\,\beta\geq 0}\mathrm{S}_{(\alpha,\beta)}$, where 
$$\mathrm{S}_{(\alpha,\beta)}=\bigoplus_{\alpha=p_1+p_2 - \sum_iq_ia_i\,;
\,\beta=\sum_iq_i} \mathbb{C}\cdot {z_{1,1}}^{p_1}{z_{1,2}}^{p_2}{z_{2,1}}^{q_1}\dots\,{z_{2,n}}^{q_n}.$$
In particular $\deg(z_{1,1})=\deg(z_{1,2})=(1,0)$ and $\deg(z_{2,i})=(-a_i,1)$. Thus the total coordinate rings can have some elements with effective degree and other elements without. Finally we have
$$\mathbb{F}(a_1,\dots,a_n)=\mathbb{F}(b_1,\dots,b_n)\Longleftrightarrow 
\left\{a_1,\dots,a_n\right\}=\left\{b_1+c,\dots,b_n+c\right\},$$
for some $c\in\mathbb{Z}$. For more details see \cite[pp. 14]{Re}.\\
 \item \label{exe4} \textbf{A toric surface.} Let us consider an example where $G$ and $\mathcal{A}_{n-1}(\mathbb{P}_{\Delta})$ have torsion. Let $\Delta$ be a complete simplicial fan in $\mathbb{Z}^2$ with edges along $v_1=2e_1-e_2$, $v_2=-e_1+2e_2$ and $v_3=-e_1-e_2$. Then $\mathbb{P}_{\Delta}:=\mathbb{P}_{\Delta(0,2,1)}$ is a compact orbifold toric surface. We have the exact sequence
$$0\longrightarrow M\stackrel{\iota}{\longrightarrow}\mathbb{Z}^{3}\stackrel{\pi}{\longrightarrow}
\frac{\,\,\,\mathbb{Z}^{3}}{Im(\iota)}\longrightarrow0,$$
where $\iota(a,b)=(2a-b)D_1+(-a+2b)D_2+(-a-b)D_3$. We have 
$$\mathcal{A}_{1}(\mathbb{P}_{\Delta(0,2,1)})=\frac{\,\,\,\mathbb{Z}^{3}}{Im(\iota)}=\frac{\mathbb{Z}D_1+\mathbb{Z}D_2+\mathbb{Z}D_3}
{\mathbb{Z}(2D_1-D_2-D_3)+\mathbb{Z}(-D_1+2D_2-D_3)},$$
therefore
$$\mathcal{A}_{1}(\mathbb{P}_{\Delta(0,2,1)})\simeq\frac{\mathbb{Z}D_1+\mathbb{Z}D_2}{\mathbb{Z}(3D_1-3D_2)}\simeq\frac{\mathbb{Z}D_1+\mathbb{Z}(D_1-D_2)}{\mathbb{Z}3(D_1-D_2)}\simeq\mathbb{Z}\oplus\mathbb{Z}_3,$$
and
$$\pi(aD_1+bD_2+cD_3)=\big(a+b+c,\left[2b+c\right]\big).$$
Note that $G\subset(\mathbb{C}^*)^3$ is given by 
$$G=\left\{(t_1,t_2,t_3)\,|\,t_1^2=t_2t_3\,,\,t_2^2=t_1t_3\right\}=
\left\{(\omega t,t,\omega^2t)\,|\,t\in\mathbb{C}^*\,,\,\omega^3=1\right\},$$
that is $G\simeq\mathbb{C}^*\times\mu_3$ and its action on $\mathbb{C}^{3}$ is given by 
$$(t,\omega)\cdot(z_1,z_2,z_3)=(t\omega z_1,tz_2, t\omega^{2} z_3).$$
Since $\Delta$ is simplicial and $\mathcal{A}_{1}(\mathbb{P}_{\Delta(0,2,1)})$ has rank $1$, we have $\mathcal{Z}=\left\{0\right\}$ and
$$\mathbb{P}_{\Delta(0,2,1)} = \left(\mathbb{C}^{3} - \{0\}\right) / \left(\mathbb{C}^*\times\mu_3\right).$$
The singular set of $\mathbb{P}_{\Delta(0,2,1)}$ is
$$\mathrm{Sing}(\mathbb{P}_{\Delta(0,2,1)})=\left\{\left[1,0,0\right],\left[0,1,0\right],\left[0,0,1\right]\right\}.$$
Moreover, we have $\deg(z_1)=(1,\left[0\right])$, $\deg(z_2)=(1,\left[2\right])$ and $\deg(z_3)=(1,\left[1\right])$.
Consequently the homogeneous coordinates ring of $\mathbb{P}_{\Delta(0,2,1)}$ is given by
$\mathrm{S}=\oplus_{\alpha\geq 0,\,\beta\in\mathbb{Z}_3}\mathrm{S}_{(\alpha, \beta)}$, where 
$$\mathrm{S}_{(\alpha, \beta)}=\bigoplus_{\alpha=m_1+m_2+m_3\,;\,\beta=\left[2m_2+m_3\right]} 
\mathbb{C}\cdot {z_1}^{m_1}{z_2}^{m_2}{z_3}^{m_3}.$$


\end{enumerate}

\bigskip

\section{One-dimensional foliations}

Let $\mathbb{P}_{\Delta}$ be a complete simplicial toric variety  of dimension $n$. Let $r$ be the
rank of finitely generated abelian group $\mathcal{A}_{n-1}(\mathbb{P}_{\Delta})$. 
There exists an exact sequence, known as the generalized Euler's sequence,
$$0\rightarrow \mathcal{O}_{\mathbb{P}_{\Delta}}^{\oplus
r}\rightarrow\bigoplus_{i=1}^{n+r}\mathcal{O}_{\mathbb{P}_{\Delta}}(D_i)\rightarrow
\mathcal{T}\mathbb{P}_{\Delta}\rightarrow0,$$
\noindent where
$\mathcal{T}\mathbb{P}_{\Delta}=\mathcal{H}om(\Omega_{\mathbb{P}_{\Delta}}^1,\mathcal{O}_{\mathbb{P}_{\Delta}})$ is the so-called Zariski tangent sheaf of $\mathbb{P}_{\Delta}$; see \cite{Cox1}. 
Let $i:\mathbb{P}_{\Delta, \,\mathrm{reg}} \rightarrow \mathbb{P}_{\Delta}$
be the inclusion of the regular part $\mathbb{P}_{\Delta, \,\mathrm{reg}}= \mathbb{P}_{\Delta} - \mathrm{Sing}(\mathbb{P}_{\Delta})$.
Since $\mathbb{P}_{\Delta}$ is a complex orbifold then $\mathcal{T}\mathbb{P}_{\Delta} \simeq 
i_{\ast}\mathcal{T}\mathbb{P}_{\Delta, \,\mathrm{reg}}$, where $\mathcal{T}\mathbb{P}_{\Delta, \,\mathrm{reg}}$ is the tangent sheaf of $\mathbb{P}_{\Delta, \,\mathrm{reg}}$; 
see \cite[Appendix A.2]{Cox3}.

Let $\mathcal{O}_{\mathbb{P}_{\Delta}}(d_1,\dots,d_{n+r})=\mathcal{O}_{\mathbb{P}_{\Delta}}(\sum_{i=1}^{n+r}d_{i}D_{i})$, 
where $\sum_{i=1}^{n+r}d_{i}D_{i}$ is a Weil divisor. 
Tensorizing the Euler's sequence by $\mathcal{O}_{\mathbb{P}_{\Delta}}(d_1,\dots,d_{n+r})$
we get
\begin{equation}\label{eul}
\begin{array}{ccccccccc}
0\rightarrow \mathcal{O}_{\mathbb{P}_{\Delta}}(d_1,\dots,d_{n+r})^{\oplus
r}\rightarrow\bigoplus_{i=1}^{n+r}\mathcal{O}_{\mathbb{P}_{\Delta}}(d_1,\dots,d_i+1,\dots,d_{n+r})\rightarrow
\mathcal{T}\mathbb{P}_{\Delta}(d_1,\dots,d_{n+r})\rightarrow0,
\end{array}
\end{equation}
\noindent where $\mathcal{T}\mathbb{P}_{\Delta}(d_1,\dots,d_{n+r})=
\mathcal{T}\mathbb{P}_{\Delta}\otimes\mathcal{O}_{\mathbb{P}_{\Delta}}(d_1,\dots,d_{n+r})$.

\begin{defi} \cite{Mau}
A one-dimensional holomorphic foliation $\mathcal{F}$ on $\mathbb{P}_{\Delta}$ of degree 
$\left[\sum_{i=1}^{n+r}d_{i}D_{i}\right]\in\mathcal{A}_{n-1}(\mathbb{P}_{\Delta})$
is a global section of $\mathcal{T}\mathbb{P}_{\Delta}(d_1,\dots,d_{n+r})$.
For simplicity of notation we say that $\mathcal{F}$ has degree $(d_1,\dots,d_{n+r})$.
We will consider one-dimensional holomorphic foliations whose singular scheme has codimension greater than $1$.
\end{defi}

Taking long exact cohomology sequence in \eqref{eul}, we have 
\begin{eqnarray*}
&\longrightarrow& \bigoplus_{i=1}^{n+r}H^0\Big(\mathbb{P}_{\Delta},\mathcal{O}_{\mathbb{P}_{\Delta}}(d_1,\dots,d_i+1,\dots,d_{n+r})\Big)
\stackrel{\rho}{\longrightarrow} H^0\big(\mathbb{P}_{\Delta},\mathcal{T}\mathbb{P}_{\Delta}(d_1,\dots,d_{n+r})\big) \\
&\longrightarrow& H^1\big(\mathbb{P}_{\Delta},\mathcal{O}_{\mathbb{P}_{\Delta}}(d_1,\dots,d_{n+r})\big)^{\oplus r}\longrightarrow \cdots
\end{eqnarray*}
We will consider one-dimensional holomorphic foliations in the image of the map $\rho$. For example, if
$H^1\left(\mathbb{P}_{\Delta}, \mathcal{O}_{\mathbb{P}_{\Delta}}
(d_1,\dots,d_{n+r})\right)=0$; see for instance the Demazure vanishing theorem \cite[Theorem $9.2.3$]{Cox2},
then, we have that a one-dimensional holomorphic foliation $\mathcal{F}$ on
$\mathbb{P}_{\Delta}$ of degree $(d_1,\dots,d_{n+r})$ is given by a
polynomial vector field in homogeneous coordinates of the form
$$
X=\sum_{i=1}^{n+r}P_i\frac{\partial}{\partial z_i},
$$
where $P_i$ is a polynomial of degree $(d_1,\dots,d_i+1,\dots,d_{n+r})$
for all $i=1,\dots,n+r$, modulo addition of a vector field of the form
$\sum_{i=1}^{r}g_iR_i$, where $R_1,\dots,R_r$ are the radial vector fields on $\mathbb{P}_{\Delta}$.
We say that $X$ is a quasi-homogeneous vector field. Moreover we have  
$$\mathrm{Sing}(\mathcal{F})=\pi\left(\left\{p\in\mathbb{C}^{n+r} : (R_1\wedge\cdots\wedge R_r\wedge X)(p)=0\right\}\right),$$
where $\pi:(\mathbb{C}^{n+r}-{\mathcal{Z}})/G \rightarrow \mathbb{P}_{\Delta}$ is the canonical projection; 
see for instance \cite{Brun}, \cite{Viv}, \cite{Jou} and \cite{So}.

Let $\mathcal{F}$ be a foliation on $\mathbb{P}_{\Delta}$ and $V = \{f = 0\}$ a quasi-homogeneous
hypersurface. We recall that $V$ is invariant by $\mathcal{F}$ if and only if 
$X(f) = g\cdot f$, where $X$ is a quasi-homogeneous vector field which defines $\mathcal{F}$ in homogeneous coordinates.

\subsection{Examples}

\bigskip

\begin{enumerate}
	\item \textbf{Weighted projective spaces.}
The Euler's sequence on $\mathbb{P}(\omega)$ is an exact sequence of orbibundles
$$
0\longrightarrow
\underline{\mathbb{C}}\longrightarrow\bigoplus_{i=0}^{n
}\mathcal{O}_{\mathbb{P}(\omega)}(\omega_i) \longrightarrow
T\mathbb{P}(\omega) \longrightarrow 0,
$$
where $\underline{\mathbb{C}}$ is the trivial line orbibundle on
$\mathbb{P}(\omega)$; see \cite{Ma}. Then, a one-dimensional holomorphic foliation $\mathcal{F}$
on $\mathbb{P}(\omega)$ of degree $d$ is a global section of 
$\mathcal{T}\mathbb{P}(\omega)\otimes\mathcal{O}_{\mathbb{P}(\omega)}(d)$. Here, the radial vector field is given by
$R = \omega_0 z_0 \frac{\partial}{\partial z_0}+\cdots+ \omega_n z_n \frac{\partial}{\partial z_n}$.\\
 \item \textbf{Multiprojective spaces.} The Euler's sequence over $\mathbb{P}^n$
$$0\longrightarrow \underline{\mathbb{C}}\;\;{\longrightarrow}\;\;
\mathcal{O}_{\mathbb{P}^{n}}(1)^{\oplus n+1}\;\;{\longrightarrow}\;\;T\mathbb{P}^{n}\longrightarrow 0,$$
gives, by direct summation, the exact sequence
$$0\longrightarrow \underline{\mathbb{C}^{2}}\,
{\longrightarrow}
\mathcal{O}_{\mathbb{P}^{n}\times\mathbb{P}^{m}}(1,0)^{\oplus n+1}\oplus\mathcal{O}_{\mathbb{P}^{n}\times\mathbb{P}^{m}}(0,1)^{\oplus m+1}
{\longrightarrow}\,
T(\mathbb{P}^{n}\times\mathbb{P}^{m})
\longrightarrow 0.$$
Then, a one-dimensional holomorphic foliation $\mathcal{F}$
on $\mathbb{P}^{n}\times\mathbb{P}^{m}$ of bidegree $(\alpha,\beta)$ is a global section of 
$T(\mathbb{P}^{n}\times\mathbb{P}^{m})\otimes\mathcal{O}_{\mathbb{P}^{n}\times\mathbb{P}^{m}}(\alpha,\beta)$. Here, the radial vector fields are given by \\
$R_1 = z_{1,0} \frac{\partial}{\partial z_{1,0}}+\cdots+ z_{1,n} \frac{\partial}{\partial z_{1,n}}$ and 
$R_2 = z_{2,0} \frac{\partial}{\partial z_{2,0}}+\cdots+ z_{2,m} \frac{\partial}{\partial z_{2,m}}$.\\
 \item \textbf{Rational normal scrolls.} The Euler's sequence on $\mathbb{F}(a):=\mathbb{F}(a_1,\dots,a_n)$ is
$$0 \rightarrow\mathcal{O}_{\mathbb{F}(a)}^{\oplus
2}\rightarrow\mathcal{O}_{\mathbb{F}(a)}(1,0)^{\oplus 2}\oplus\bigoplus_{i=1}^n
\mathcal{O}_{\mathbb{F}(a)}(-a_i,1)\rightarrow \mathcal{T}\mathbb{F}(a)\rightarrow0.$$
Then, a one-dimensional holomorphic foliation $\mathcal{F}$
on $\mathbb{F}(a)$ of bidegree $(d_1,d_2)$ is a global section of 
$T\mathbb{F}(a)\otimes \mathcal{O}_{\mathbb{F}(a)}(d_1,d_2)$. Here, the radial vector fields are given by \\
$R_1 = z_{1,1}\frac{\partial}{\partial z_{1,1}}+ z_{1,2}\frac{\partial}{\partial z_{1,2}} 
+ \sum_{i=1}^{n}-a_i z_{2,i}\frac{\partial}{\partial z_{2,i}}$ and 
$R_2 = \sum_{i=1}^{n}z_{2,i}\frac{\partial}{\partial z_{2,i}}$.\\
 \item \textbf{A toric surface.} The Euler's sequence on $\mathbb{P}_{\Delta(0,2,1)}$ is given by
$$0 \rightarrow \mathcal{O}_{\mathbb{P}_{\Delta(0,2,1)}}\rightarrow\mathcal{O}_{\mathbb{P}_{\Delta(0,2,1)}}(1,\left[0\right])\oplus
\mathcal{O}_{\mathbb{P}_{\Delta(0,2,1)}}(1,\left[2\right])\oplus\mathcal{O}_{\mathbb{P}_{\Delta(0,2,1)}}(1,\left[1\right])
\rightarrow\mathcal{T}\mathbb{P}_{\Delta(0,2,1)}\rightarrow0.$$
Then, a one-dimensional holomorphic foliation $\mathcal{F}$
on $\mathbb{P}_{\Delta(0,2,1)}$ of bidegree $(\alpha,\beta)$ is a global section of 
$\mathcal{T}\mathbb{P}_{\Delta(0,2,1)}\otimes\mathcal{O}_{\mathbb{P}_{\Delta(0,2,1)}}(\alpha,\beta)$. Here, the radial vector field is given by
$R = z_1 \frac{\partial}{\partial z_1}+ z_2 \frac{\partial}{\partial z_2}+ z_3 \frac{\partial}{\partial z_3}$.\\
\end{enumerate}

\section{Poincar\'e problem}

Let $\mathbb{P}_{\Delta}=(\mathbb{C}^{n+r}-{\mathcal{Z}})/G$ be a complete simplicial toric variety  of dimension $n$.

\begin{defi}
Let $f\in \mathbb{C}[z_1,\dots,z_{n+r}]$ be a quasi-homogeneous polynomial. We say that $V=\left\{f=0\right\}\subset\mathbb{P}_{\Delta}$ is 
\emph{strongly quasi-smooth} if its tangent cone $\left\{f = 0\right\}$ on $\mathbb{C}^{n+r} - \{0\}$ is smooth. Note that strongly quasi-smooth implies quasi-smooth.
\end{defi}

Consider the $r$ linearly independent over $\mathbb{Z}$ relations among the $n_{\rho_1},\dots,n_{\rho_{n+r}}$
$$\sum_{j=1}^{n+r}a_{i,j}n_{\rho_j}=0,\,\,\, i=1,\dots,r.$$

\begin{teo}[Normal form 1]\label{Mig1}
Let $\mathbb{P}_{\Delta}$ be a complete simplicial toric variety of dimension $n$, with homogeneous coordinates 
$z_1,\ldots,z_{n+r}$.
Let $V=\left\{f=0\right\}\subset\mathbb{P}_{\Delta}$ be a strongly quasi-smooth hypersurface of degree $\alpha\in\mathcal{A}_{n-1}(\mathbb{P}_{\Delta})$. If $X$ is a quasi-homogeneous vector field that leaves $V$ invariant, then
$$X = \sum_{j<k}P_{j,k}^{\,i}
\left(\frac{\partial f}{\partial z_{j}}\frac{\partial}{\partial z_{k}}-
\frac{\partial f}{\partial z_{k}}\frac{\partial}{\partial z_{j}} \right)+
\frac{g}{\theta_i(\alpha)}\sum_{j=1}^{n+r}a_{i,j} z_{j} \frac{\partial }{\partial z_{j}}\,,\,i=1,\dots,r.$$
where $P_{j,k}^{\,i}, \, g \in \mathbb{C}[z_1,\dots,z_{n+r}]$ are quasi-homogeneous polynomials. 
Here, $\theta_i(\alpha)$ is a complex number defined in Subsection $\ref{Qsh}$.
\end{teo}

\begin{proof}
We use the Koszul complex following the ideas of Zariski-Esteves \cite{Est}.
Set $X=\sum_{j=1}^{n+r} P_{j}\frac{\partial}{\partial z_{j}}$ and put 
$E_{\bullet}=\mathbb{C}[z]\otimes_{\mathbb{C}}\wedge^{\bullet}\mathbb{C}^{n+r}$, where $\mathbb{C}[z]=\mathbb{C}[z_1,\dots,z_{n+r}]$. 
Euler's generalized formula implies that 
$i_{R_i}(df)=\theta_i (\alpha) \cdot f$ for all $i=1,\dots,r$. That is
$$\theta_{i} (\alpha) \cdot f = \sum_{j=1}^{n+r} a_{i,j} z_{j} \frac{\partial f}{\partial z_{j}}\,,\,i=1,\dots,r.$$

\noindent The invariance of $V$ implies that 
$$g\cdot f = \sum_{j=1}^{n+r} P_{j} \frac{\partial f}{\partial z_{j}},$$

\noindent for some polynomial $g \in \mathbb{C}[z]$. 
From these two equations we obtain the following polynomial relationship    
$$\sum_{j=1}^{n+r} \Big(P_{j} - \frac{g}{\theta_{i}(\alpha)} \, a_{i,j} z_{j}\Big)
 \frac{\partial f}{\partial z_{j}}=0\,,\,i=1,\dots,r.$$

\noindent This identity says that the vector fields
$$X'_i=X-\frac{g}{\theta_i(\alpha)}\sum_{j=1}^{n+r}a_{i,j} z_{j} \frac{\partial }{\partial z_{j}}\,,\,i=1,\dots,r$$

\noindent satisfies $\partial_1(X'_i)=0$, that is, $X'_i \in Ker(\partial_1)$, where $\partial_1:E_1\longrightarrow E_0$ 
is the Koszul complex $\mathbb{C}[z]$-linear operator associated to
$\textsl{S}=(\frac{\partial f}{\partial z_{1}},\dots,\frac{\partial f}{\partial z_{n+r}})$, 
given by $\partial_1(\frac{\partial}{\partial z_{k}})=\frac{\partial f}{\partial z_{k}}$ for $k=1,\ldots,n+r$; see \cite[pp. 688]{Grif}. 
By hypothesis, the singular set of the hypersurface consists of 
$$\Big\{\frac{\partial f}{\partial z_{1}}=\cdots=\frac{\partial f}{\partial z_{n+r}}=0\Big\}=\{0\}.$$
Then $\textsl{S}$ is a regular sequence and consequently
$H_1(E_{\bullet}(\textsl{S}))=0$, i.e. 
$Ker(\partial_1)=Im(\partial_2)$, where $\partial_2:E_2\longrightarrow E_1$ 
is the Koszul complex $\mathbb{C}[z]$-linear operator associated to $\textsl{S}$, 
given by $\partial_2(\frac{\partial}{\partial z_{i}}\wedge\frac{\partial}{\partial z_{j}})=
\frac{\partial f}{\partial z_{i}}\frac{\partial}{\partial z_{j}}-\frac{\partial f}{\partial z_{j}}\frac{\partial}{\partial z_{i}}$ 
for $1\leq i<j\leq n+r$; see \cite[pp. 688, Lemma]{Grif}. Therefore, there exist $P_{j,k}^{\,i}\in\mathbb{C}[z]$ such that
$$X'_i=X - \frac{g}{\theta_i(\alpha)}\sum_{j=1}^{n+r}a_{i,j} z_{j} \frac{\partial }{\partial z_{j}}=
\sum_{j<k}P_{j,k}^{\,i}  
\left( \frac{\partial f}{\partial z_{j}}\frac{\partial}{\partial z_{k}}-
\frac{\partial f}{\partial z_{k}}\frac{\partial}{\partial z_{j}} \right),\,i=1,\dots,r.$$
Hence 
$$X = \sum_{j<k}P_{j,k}^{\,i}
\left( \frac{\partial f}{\partial z_{j}}\frac{\partial}{\partial z_{k}}-
\frac{\partial f}{\partial z_{k}}\frac{\partial}{\partial z_{j}} \right)+
\frac{g}{\theta_i(\alpha)}\sum_{j=1}^{n+r}a_{i,j} z_{j} \frac{\partial }{\partial z_{j}}\,,\,i=1,\dots,r.$$
\end{proof}

The Weil divisor class group $\mathcal{A}_{n-1}(\mathbb{P}_{\Delta})$ is a finitely generated abelian group of rank $r$. By the fundamental theorem of finitely generated abelian groups, we have
\begin{equation}\label{eq003}
\begin{array}{ccccccccc}
\mathcal{A}_{n-1}(\mathbb{P}_{\Delta})\simeq\mathbb{Z}^{\,r} \oplus \mathbb{Z}_{(p_1)^{\lambda_1}} \oplus \cdots \oplus \mathbb{Z}_{(p_m)^{\lambda_m}},
\end{array}
\end{equation}
where $p_i$ are primes, not necessarily distinct, and $\lambda_i$ are positive integers. The 
direct sum is unique except for possible rearrangement of the factors. 

Suppose that the homogeneous coordinate ring $\mathrm{S}$ of $\mathbb{P}_{\Delta}$ has the following property: there is a positive integer number 
$1 \leq k \leq r$ such that
\begin{equation}\label{eq3}
\begin{array}{ccccccccc}
(\deg \mathrm{P})_k \geq 0\,\,\textup{for all}\,\,\mathrm{P} \in \mathrm{S},
\end{array}
\end{equation}
where $(\deg \mathrm{P})_k$ denotes the $k$-th integer coordinate of $\deg\mathrm{P}\in\mathcal{A}_{n-1}(\mathbb{P}_{\Delta})$,
i.e. the $k$-th component of the multidegree $\deg\mathrm{P}$ in \eqref{eq003}. The first main result is the following:

\begin{teo}\label{Mig2} Let $\mathbb{P}_{\Delta}$ be a complete simplicial toric variety of dimension $n$, with homogeneous coordinates 
$z_1,\ldots,z_{n+r}$. Let $\mathcal{F}$ be a one-dimensional holomorphic foliation on $\mathbb{P}_{\Delta}$ and let $X$ be a quasi-homogeneous vector field which defines $\mathcal{F}$ in homogeneous coordinates. Let $V=\left\{f=0\right\}\subset\mathbb{P}_{\Delta}$ be a strongly quasi-smooth hypersurface invariant by $\mathcal{F}$. Then 
$$\deg(V)_k \leq \deg(\mathcal{F})_k + \max_{1\leq i<j\leq n+r} \big\{\deg(z_i)_k + \deg(z_j)_k\big\},$$
for each $1\leq k \leq r$ as in the hypothesis $(\ref{eq3})$.
\end{teo}

\begin{proof}
By Theorem \ref{Mig1}, the quasi-homogeneous vector field $X$ that defines $\mathcal{F}$ is given by 
$$X = \sum_{j<k}P_{j,k}^{\,i}
\left(\frac{\partial f}{\partial z_{j}}\underline{\frac{\partial}{\partial z_{k}}}-
\frac{\partial f}{\partial z_{k}}\frac{\partial}{\partial z_{j}} \right)+
\frac{g}{\theta_i(\alpha)}\sum_{j=1}^{n+r}a_{i,j} z_{j} \frac{\partial }{\partial z_{j}}\,,\,i=1,\dots,r.$$
where $P_{j,k}^{\,i}, \, g \in \mathbb{C}[z_1,\dots,z_{n+r}]$ are quasi-homogeneous polynomials.
As $X\notin \mathrm{Lie}(G)=\langle R_1,\dots,R_r\rangle$, we have $P_{j,k}^{\,i}\neq 0$ for some $j<k$. 
Set $\deg(\mathcal{F})=\sum_{i=1}^{n+r} d_i\left[D_i\right]$. Then 
$$\deg\Big(P_{j,k}^{\,i}\cdot \frac{\partial f}{\partial z_j}\Big)=(d_k+1)\left[D_k\right]+\sum_{i\neq k}d_i\left[D_i\right].$$
Therefore
\begin{eqnarray*}
\deg(P_{j,k}^{\,i})
&=&(d_k+1)\left[D_k\right]+\sum_{i\neq k}d_i\left[D_i\right]-\deg \big(\frac{\partial f}{\partial z_j}\big) \\
&=&(d_k+1)\left[D_k\right]+\sum_{i\neq k}d_i\left[D_i\right]-\big(\deg(V)-\left[D_j\right]\big)\\
&=&\sum_i d_i\left[D_i\right]+\left[D_k\right]+\left[D_j\right]-\deg(V)\\
&=&\deg(\mathcal{F})+\left[D_j\right]+\left[D_k\right]-\deg(V)\\
&=&\deg(\mathcal{F})+\deg(z_j)+\deg(z_k)-\deg(V).
\end{eqnarray*}
Let $1\leq \ell \leq r$ be as in the hypothesis $(\ref{eq3})$. Then 
$$0\leq\deg(P_{j,k}^{\,i})_{\ell}=\deg(\mathcal{F})_{\ell}+\deg(z_j)_{\ell}+\deg(z_k)_{\ell}-\deg(V)_{\ell}.$$
Finally, we have
$$\deg(V)_{\ell}\leq\deg(\mathcal{F})_{\ell}+\deg(z_j)_{\ell}+\deg(z_k)_{\ell}\leq
\deg(\mathcal{F})_{\ell} + \max_{1\leq j<k\leq n+r} \big\{\deg(z_j)_{\ell} + \deg(z_k)_{\ell}\big\}.$$
\end{proof}

\begin{obs} \label{Mig3} In Theorem \ref{Mig2}
\begin{enumerate}
	\item [(i)] A more precise conclusion is
	$\deg(V)_k \leq \deg(\mathcal{F})_k + \deg(z_i)_k + \deg(z_j)_k$, for some $i<j$ and for each $1\leq k \leq r$ as in the hypothesis
	$(\ref{eq3})$.
	\item [(ii)] Suppose that the homogeneous coordinate ring $\mathrm{S}$ of $\mathbb{P}_{\Delta}$ has elements with arbitrary degree, 
then we only have $\deg(P_{j,k}^{\,i})=\deg(\mathcal{F})+\deg(z_j)+\deg(z_k)-\deg(V)$, for some $j<k$.
\end{enumerate}
\end{obs}

\bigskip


There is a situation in which the above Theorem \ref{Mig2} is valid for a quasi-smooth hypersurface 
$V=\left\{f=0\right\}\subset\mathbb{P}_{\Delta}$ with
$$\left\{0\right\} \subsetneq \mathrm{Sing}(V) \subset \mathcal{Z}.$$
It is easy to understand that the same arguments in the proof of Theorems \ref{Mig1} and \ref{Mig2} 
can be applied here to obtain the Theorems \ref{Mig4} and \ref{Mig5} below

\begin{teo}[Normal form 2]\label{Mig4}
Let $\mathbb{P}_{\Delta}$ be a complete simplicial toric variety of dimension $n$, with homogeneous coordinates $z_1,\ldots,z_{n+r}$.
Let $V=\left\{f=0\right\}\subset\mathbb{P}_{\Delta}$ be a quasi-smooth hypersurface of degree 
$\alpha\in\mathcal{A}_{n-1}(\mathbb{P}_{\Delta})$. Suppose there are integer numbers $1\leq i_1<\cdots<i_k\leq n+r$ such that 
\begin{enumerate}
 	\item [(i)] there is a regular subsequence 
$\left\{\frac{\partial f}{\partial z_{i_1}},\dots,\frac{\partial f}{\partial z_{i_k}}\right\}\subset
\left\{\frac{\partial f}{\partial z_{1}},\dots,\frac{\partial f}{\partial z_{n+r}}\right\}$ and
  \item [(ii)] there is a radial vector field $R_{i_1,\ldots,i_k}=\sum_{j=1}^{k} a_{i_j}z_{i_j}\frac{\partial}{\partial z_{i_j}}$ such that
$i_{R_{i_1,\ldots,i_k}}(df)=\theta(\alpha)\cdot f$, where $\alpha\in\mathcal{A}_{n-1}(\mathbb{P}_{\Delta})$ is the degree of $f$ and 
$\theta(\alpha)$ is a constant.
\end{enumerate}
Now, let $X$ be a quasi-homogeneous vector field in $\mathbb{P}_{\Delta}$ and 
suppose that, in homogeneous coordinates, $X=X_1+X_2$, where $X_1=\sum_{j=1}^{k} P_{i_j}\frac{\partial}{\partial z_{i_j}}$
leaves $V$ invariant. Then
$$X_1 = \sum_{j_1<j_2}P_{j_1,j_2}
\left(\frac{\partial f}{\partial z_{i_{j_1}}}\frac{\partial}{\partial z_{i_{j_2}}}-
\frac{\partial f}{\partial z_{i_{j_2}}}\frac{\partial}{\partial z_{i_{j_1}}} \right)+
\frac{g}{\theta(\alpha)}\sum_{j=1}^{k}a_{i_j} z_{i_j} \frac{\partial }{\partial z_{i_j}},$$
where $P_{j_1,j_2}, \, g \in \mathbb{C}[z_1,\dots,z_{n+r}]$ are quasi-homogeneous polynomials.
\end{teo}

The property $($i$)$ is equivalent to 
$$\mathrm{codim}\left( \left\{\frac{\partial f}{\partial z_{i_1}}=\cdots=\frac{\partial f}{\partial z_{i_s}}=0\right\}\right)
=s,\,s\leq k;$$
for more details see \cite{Grif}. 

\begin{proof}
Set $E_{\bullet}=\mathbb{C}[z]\otimes_{\mathbb{C}}\wedge^{\bullet}\mathbb{C}^{k}$, where $\mathbb{C}[z]=\mathbb{C}[z_{1},\dots,z_{n+r}]$. 
The hypothesis $(\mathrm{ii})$ implies that 
$$\theta (\alpha) \cdot f = \sum_{j=1}^{k} a_{i_j} z_{i_j} \frac{\partial f}{\partial z_{i_j}}.$$

\noindent Now, the invariance of $V$ by $X_1$ implies that 
$$g\cdot f = \sum_{j=1}^{k} P_{i_j} \frac{\partial f}{\partial z_{i_j}},$$

\noindent for some polynomial $g \in \mathbb{C}[z]$. 
From these two equations we obtain the following polynomial relationship    
$$\sum_{j=1}^{k} \Big(P_{i_j} - \frac{g}{\theta(\alpha)} \, a_{i_j} z_{i_j}\Big)
 \frac{\partial f}{\partial z_{i_j}}=0.$$

\noindent This identity says that the vector fields
$$X'_1=X_1-\frac{g}{\theta(\alpha)}\sum_{j=1}^{k}a_{i_j} z_{i_j} \frac{\partial }{\partial z_{i_j}}$$

\noindent satisfies $\partial_1(X'_1)=0$, that is, $X'_1 \in Ker(\partial_1)$, where $\partial_1$ 
is the Koszul complex operator associated to
$\textsl{S}=(\frac{\partial f}{\partial z_{i_1}},\dots,\frac{\partial f}{\partial z_{i_k}})$. 
By the hypothesis $(\mathrm{i})$, we have $H_1(E_{\bullet}(\textsl{S}))=0$, i.e. 
$Ker(\partial_1)=Im(\partial_2)$, where $\partial_2$ 
is the Koszul complex operator associated to $\textsl{S}$. 
Therefore, there exist $P_{j_1,j_2}\in\mathbb{C}[z]$ such that
$$X'_1=X_1 - \frac{g}{\theta(\alpha)}\sum_{j=1}^{k}a_{i_j} z_{i_j} \frac{\partial }{\partial z_{i_j}}=
\sum_{j_1<j_2}P_{j_1,j_2} 
\left( \frac{\partial f}{\partial z_{i_{j_1}}}\frac{\partial}{\partial z_{i_{j_2}}}-
\frac{\partial f}{\partial z_{i_{j_2}}}\frac{\partial}{\partial z_{i_{j_1}}} \right).$$
Hence 
$$X_1 = \sum_{j_1<j_2}P_{j_1,j_2}
\left(\frac{\partial f}{\partial z_{i_{j_1}}}\frac{\partial}{\partial z_{i_{j_2}}}-
\frac{\partial f}{\partial z_{i_{j_2}}}\frac{\partial}{\partial z_{i_{j_1}}} \right)+
\frac{g}{\theta(\alpha)}\sum_{j=1}^{k}a_{i_j} z_{i_j} \frac{\partial }{\partial z_{i_j}}.$$
\end{proof}

The second main result is the following:

\begin{teo}\label{Mig5} Let $\mathbb{P}_{\Delta}$ be a complete simplicial toric variety of dimension $n$, with homogeneous coordinates 
$z_1,\ldots,z_{n+r}$.
Let $\mathcal{F}$ be a one-dimensional holomorphic foliation on
$\mathbb{P}_{\Delta}$ and let $X$ be a quasi-homogeneous vector field which defines $\mathcal{F}$ in homogeneous coordinates. 
Let $V=\left\{f=0\right\}\subset\mathbb{P}_{\Delta}$ be a quasi-smooth hypersurface. Suppose there are integers numbers
$1\leq i_1<\cdots<i_k\leq n+r$ with the properties $($\emph{i}$)$ and $($\emph{ii}$)$ as in the above Theorem $\ref{Mig4}$.
Moreover assume that, in homogeneous coordinates, $X=X_1+X_2$, where $X_1=\sum_{j=1}^{k} P_{i_j}\frac{\partial}{\partial z_{i_j}}$
leaves $V$ invariant and such that $X_1\notin \mathrm{Lie}(G)$. Then              
$$\deg(V)_{\ell} \leq \deg(\mathcal{F})_{\ell} + \max_{1\leq j_1< j_2\leq k} \big\{\deg(z_{i_{j_1}})_{\ell} 
+ \deg(z_{i_{j_2}})_{\ell}\big\},$$
for each $1\leq \ell \leq r$ as in the hypothesis $(\ref{eq3})$.
\end{teo}

\begin{proof}
By Theorem \ref{Mig4}, the quasi-homogeneous vector field $X_1$ is given by 
$$X_1 = \sum_{j_1<j_2}P_{j_1,j_2}
\left(\frac{\partial f}{\partial z_{i_{j_1}}}\underline{\frac{\partial}{\partial z_{i_{j_2}}}}-
\frac{\partial f}{\partial z_{i_{j_2}}}\frac{\partial}{\partial z_{i_{j_1}}} \right)+
\frac{g}{\theta(\alpha)}\sum_{j=1}^{k}a_{i_j} z_{i_j} \frac{\partial }{\partial z_{i_j}}.$$
where $P_{j_1,j_2}, \, g \in \mathbb{C}[z_1,\dots,z_{n+r}]$ are quasi-homogeneous polynomials.
As $X_1\notin \mathrm{Lie}(G)$, we have $P_{j_1,j_2}\neq 0$ for some $j_1<j_2$. 
Set $\deg(\mathcal{F})=\sum_{i=1}^{n+r} d_i\left[D_i\right]$. Then 
$$\deg\Big(P_{j_1,j_2}\cdot \frac{\partial f}{\partial z_{i_{j_1}}}\Big)=
\left(d_{i_{j_2}}+1\right)\left[D_{i_{j_2}}\right]+\sum_{i\neq i_{j_2}}d_i\left[D_i\right].$$
We can now proceed analogously to the proof of Theorem \ref{Mig2}, so we have 
$$\deg(P_{j_1,j_2})=\deg(\mathcal{F})+\deg(z_{i_{j_1}})+\deg(z_{i_{j_2}})-\deg(V).$$
Let $1\leq \ell \leq r$ be as in the hypothesis $(\ref{eq3})$. Then 
$$0\leq\deg(P_{j_1,j_2})_{\ell}=\deg(\mathcal{F})_{\ell}+\deg(z_{i_{j_1}})_{\ell}+\deg(z_{i_{j_2}})_{\ell}-\deg(V)_{\ell}.$$
Finally, we have
$$\deg(V)_{\ell}\leq\deg(\mathcal{F})_{\ell}+\deg(z_{i_{j_1}})_{\ell}+\deg(z_{i_{j_2}})_{\ell}\leq
\deg(\mathcal{F})_{\ell} + \max_{1\leq j_1< j_2\leq k} \big\{\deg(z_{i_{j_1}})_{\ell} 
+ \deg(z_{i_{j_2}})_{\ell}\big\}.$$
\end{proof}

It is worth mentioning that an analogous remark to Remark \ref{Mig3} is valid here


\bigskip

\section{Applications} \label{App}

In this section we use Theorem \ref{Mig2} to give an upper bound for the degree of a strongly quasi-smooth hypersurface invariant by 
$\mathcal{F}$ on weighted projective spaces, multiprojective spaces, rational normal scrolls, and for a compact toric orbifold surface with Weil divisor class group having torsion $\mathbb{P}_{\Delta(0,2,1)}$. According to Theorem \ref{Mig5}, it is clear that analogous results can be obtained for quasi-smooth hypersurfaces. We build several families of one-dimensional foliations where the upper bound is sharp; see Examples \ref{Wps2}, \ref{Ms2}, \ref{Tor2}, \ref{A} and \ref{B}.

\begin{cor}[Weighted projective spaces]\label{Wps1} Let $\mathcal{F}$ be a one-dimensional holomorphic foliation on
$\mathbb{P}(\omega)$ and let $X$ be a quasi-homogeneous vector field which defines $\mathcal{F}$ in homogeneous coordinate. 
Let $V \subset \mathbb{P}(\omega)$ be a quasi-smooth hypersurface invariant by $\mathcal{F}$. Then  
$$\deg(V) \leq \deg(\mathcal{F}) + \max_{0\leq i<j\leq n} \{\omega_i + \omega_j\}.$$
\end{cor}

\begin{proof} 
Follow directly from Theorem \ref{Mig2}, because $\mathcal{A}_{n-1}(\mathbb{P}_{\omega})\simeq\mathbb{Z}$, 
$\deg(z_i)=\omega_i$ for all $0\leq i\leq n$ and the one-dimensional holomorphic foliation $\mathcal{F}$
on $\mathbb{P}(\omega)$ is a global section of  
$\mathcal{T}\mathbb{P}(\omega)\otimes\mathcal{O}_{\mathbb{P}(\omega)}(d)$. 
\end{proof}

Now, let us consider the family of examples of one-dimensional holomorphic foliations in \cite{MiMaFa}

\begin{exe}\label{Wps2}
Let $\omega_0,\omega_1,\dots ,\omega_{2m+1}$ be positive integers with $gcd(\omega_0,\dots,\omega_{2m+1})=1$ and $\xi$  such that
$$\xi=\omega_{2j}+\omega_{2j+1}\,\,\textup{for all}\,\,j=0,1,\dots,m.$$
Let us consider the weighted projective space $\mathbb{P}^{n}(\omega_{0},\dots,\omega_{n})$, where $n=2m+1$
and $\mathcal{F}$ a holomorphic foliation on $\mathbb{P}^{n}(\omega_{0}\dots,\omega_{n})$ induced by the quasi-homogeneous vector field with isolated singularities given by 
$$X=\sum_{k=0}^{m}\Big(d_{2k+1}z_{2k+1}^{d_{2k+1}-1}\frac{\partial}{\partial z_{2k}}-d_{2k}z_{2k}^{d_{2k}-1}\frac{\partial}{\partial z_{2k+1}}\Big),$$
where  $d_{0},\dots, d_{n} \in \mathbb{N}$ and $\zeta$ satisfy the relation
$$\zeta=\omega_{k}d_{k}\,\,\textup{for all}\,\,k=0,\dots,n.$$	
For each $c=(c_0:\cdots:c_m)\in \mathbb{P}^m$,  $V_c$ is the quasi-smooth hypersurface on 
$\mathbb{P}^{n}(\omega_{0},\dots,\omega_{n})$ of degree $\zeta$ given by 
$$
V_{c}=\left\{\sum_{k=0}^{m}c_k\big(z_{2k}^{d_{2k}}+z_{2k+1}^{d_{2k+1}}\big)=0\right\}.
$$
We can see that $V_{c}$ is invariant by $\mathcal{F}$ and $\deg (\mathcal{F}) = \zeta - \xi$. Then 
$$\deg (V_c)-\deg(\mathcal{F})= \xi \leq \max_{0\leq i<j\leq n} \{\omega_i + \omega_j\}.$$

Consider a similar foliation on even dimensional weighted projective spaces $\mathbb{P}^{n}(\omega_{0},\dots,\omega_{n})$ where 
$n=2m +2$ and $\xi=\omega_{2k}+\omega_{2k+1}\,\,\textup{for all}\,\,k=0,\dots,m$. Suppose that 
$\zeta=\omega_{k}d_{k}$ for all $\, k=0,\dots,n$ and consider the vector field $X$ of the previous example. Thus, for each 
$c=(c_0:\cdots:c_{m+1})\in \mathbb{P}^{m+1}$, the quasi-smooth hypersurface on $ \mathbb{P}^{n}(\omega_{0},\dots,\omega_{n})$ 
of degree $\zeta$ given by 
$$
V_{c}= \left\{\sum_{k=0}^{m}c_k\big(z_{2k}^{d_{2k}}+z_{2k+1}^{d_{2k+1}}\big)+c_{m+1}z_{n}^{d_{n}}=0\right\}
$$
is invariant by $X$ and therefore we obtain the same conclusions 
$$\deg(V_c) - \deg (\mathcal F) = \xi \leq \max_{0\leq i<j\leq n} \{\omega_i + \omega_j\}.$$

We see that the upper bound for Poincar\'e problem is optimal. Compare with Remark \ref{Mig3}-(i).
\end{exe}

\begin{cor}[Multiprojective spaces]\label{Ms1} Let $\mathcal{F}$ be a one-dimensional holomorphic foliation on
$\mathbb{P}^{n_1}\times\cdots\times\mathbb{P}^{n_r}$ $(r>1)$ and let $X$ be a quasi-homogeneous vector field which defines $\mathcal{F}$ in homogeneous coordinates. Let $V \subset \mathbb{P}^{n_1}\times\cdots\times\mathbb{P}^{n_r}$ be a strongly 
quasi-smooth hypersurface invariant by $\mathcal{F}$. Then 
$$\deg(V)_k \leq \deg(\mathcal{F})_k + 2,\,\,1\leq k \leq r.$$
\end{cor}

\begin{proof}
In homogeneous coordinates $z=(z_1,\dots,z_r)\in\mathbb{C}^{(n_1+\cdots+n_r)+r}$ with $z_i=(z_{i,0},\dots,z_{i,n_i})\in\mathbb{C}^{n_{i}+1}$,
we have $\deg(z_{i,j})=\left(0,\dots,1_i,\dots,0\right)$, for all $0\leq j \leq n_i$. Then the corollary follows directly from Theorem 
\ref{Mig2}, because $\mathcal{A}_{n_1+\cdots+n_r-1}(\mathbb{P}^{n_1}\times\cdots\times\mathbb{P}^{n_r})\simeq\mathbb{Z}^r$ and one-dimensional holomorphic foliation $\mathcal{F}$ on $\mathbb{P}^{n_1}\times\cdots\times\mathbb{P}^{n_r}$ of multidegree $(\alpha_1,\dots,\alpha_r)$
is a global section of $T(\mathbb{P}^{n_1}\times\cdots\times\mathbb{P}^{n_r})
\otimes\mathcal{O}_{\mathbb{P}^{n_1}\times\cdots\times\mathbb{P}^{n_r}}(\alpha_{1},\dots,\alpha_{r})$.
\end{proof}

Now, let us consider a family of examples of one-dimensional holomorphic foliations in $\mathbb{P}^{n}\times\mathbb{P}^{n}$

\begin{exe}\label{Ms2}
In homogeneous coordinates $z_{1,0},z_{1,1},\ldots z_{1,n}, z_{2,0},z_{2,1},\ldots z_{2,n}$ with $n=2m+1$, let us consider the multiprojective space $\mathbb{P}^{n}\times\mathbb{P}^{n}$ and $\mathcal{F}$ a holomorphic foliation on $\mathbb{P}^{n}\times\mathbb{P}^{n}$ induced by the quasi-homogeneous vector field
$$X=\sum_{k=0}^{m}a_kz_{1,2k}^2\Big(z_{2,2k+1}\frac{\partial}{\partial z_{1,2k}}-z_{2,2k}\frac{\partial}{\partial z_{1,2k+1}}\Big)+
\sum_{k=0}^{m}b_kz_{2,2k}^2\Big(z_{1,2k+1}\frac{\partial}{\partial z_{2,2k}}-z_{1,2k}\frac{\partial}{\partial z_{2,2k+1}}\Big),$$
where $a=(a_0:\ldots:a_m), b=(b_0:\ldots:b_m)\in\mathbb{P}^{m}$.
Consider the strongly quasi-smooth hypersurface on $\mathbb{P}^n\times\mathbb{P}^n$ invariant by $\mathcal{F}$ given by 
$$V=\left\{\sum_{k=0}^{n}z_{1,k}z_{2,k}=0\right\}.$$
It is easy see that $\deg(V)=\deg(\mathcal{F}) = (1,1)$ and so 
$$\deg(V)_k \leq \deg(\mathcal{F})_k + 2,\,\,k=1, 2.$$
\end{exe}

\begin{cor}[Rational normal scrolls]\label{Rs} Let $\mathcal{F}$ be a one-dimensional holomorphic foliation on
$\mathbb{F}(a_1,\dots,a_n)$ and let $X$ be a quasi-homogeneous vector field which defines $\mathcal{F}$ in homogeneous coordinates. 
Let $V\subset\mathbb{F}(a_1,\dots,a_n)$ be a strongly 
quasi-smooth hypersurface invariant by $\mathcal{F}$. Then
\begin{enumerate}
	\item [(i)] if $n=1$, then $\deg(V)_2\leq\deg(\mathcal{F})_2+1,$
	\item [(ii)] if $n>1$, then $\deg(V)_2\leq\deg(\mathcal{F})_2+2.$
\end{enumerate}
Moreover, suppose that $a_1,\dots,a_n$ are non-positive integers, we have 
\begin{enumerate}
	\item [(i)] if $a_1=a_2=\cdots=a_n=0$, then $\deg(V)_1\leq\deg(\mathcal{F})_1+2,$
	\item [(ii)] if $a_2=\cdots=a_n=0$ and $a_1$ is negative, then $\deg(V)_1\leq\deg(\mathcal{F})_1+1-a_1,$
	\item [(iii)] if at least two of the $a_i$ are negatives, then 
$\deg(V)_1\leq\deg(\mathcal{F})_1-\min_{1\leq i<j \leq n}\left\{a_i+a_j\right\}.$
\end{enumerate}
\end{cor}

\begin{proof} Here $\mathcal{A}_{n-1}(\mathbb{F}(a_1,\dots,a_n))=\mathbb{Z}^2$ and every bihomogeneous polynomial of bidegree 
$(\alpha,\beta)$ has $\beta\geq 0$. Then, the corollary follows directly from Theorem \ref{Mig2} and Example $3.4$-(\ref{exe3}), because 
$$\deg(V)_2\leq\deg(\mathcal{F})_2+\max_{1\leq i<j \leq n+2}\left\{\deg(z_i)_2+\deg(z_j)_2\right\}.$$
Moreover, if $a_1,\dots,a_n$ are non-positive integers, the corollary follows directly from Theorem \ref{Mig2} and Example 
$3.4$-(\ref{exe3}), because $\alpha\geq 0$ and 
$$\deg(V)_1\leq\deg(\mathcal{F})_1+\max_{1\leq i<j \leq n+2}\left\{\deg(z_i)_1+\deg(z_j)_1\right\}.$$
\end{proof}

\begin{cor}[A toric surface]\label{Tor1} Let $\mathcal{F}$ be a one-dimensional holomorphic foliation on $\mathbb{P}_{\Delta(0,2,1)}$ and let 
$X$ be a quasi-homogeneous vector field which defines $\mathcal{F}$ in homogeneous coordinates. 
Let $V\subset\mathbb{P}_{\Delta(0,2,1)}$ be a quasi-smooth hypersurface invariant by $\mathcal{F}$. Then
$$\deg(V)_1\leq\deg(\mathcal{F})_1+2.$$
\end{cor}

\begin{proof} Here $\mathcal{A}_{1}(\mathbb{P}_{\Delta(0,2,1)})=\mathbb{Z}\oplus\mathbb{Z}_3$ and every bihomogeneous polynomial of bidegree 
$(\alpha,\beta)$ has $\alpha\geq 0$. Then, the corollary follows directly from Theorem \ref{Mig2}.
\end{proof}

Now, let us consider a family of examples of one-dimensional holomorphic foliations in $\mathbb{P}_{\Delta(0,2,1)}$

\begin{exe}\label{Tor2}
Consider a toric surface $\mathbb{P}_{\Delta(0,2,1)}$. Let $m$ be a positive integer with $m\equiv 0\,\textup{mod}\,3$ and $\mathcal{F}$ a holomorphic foliation on $\mathbb{P}_{\Delta(0,2,1)}$, induced by the quasi-homogeneous vector field
$$X = z_2^m \frac{\partial}{\partial z_1}+ z_1(z_3^{m-1}-z_1^{m-2}z_2) \frac{\partial}{\partial z_2}-
z_1z_2^{m-1} \frac{\partial}{\partial z_3}.$$
We have $\mathrm{Sing}(\mathcal{F})=\left\{(1:a:-a^{-1})\in(\mathbb{C^*})^{3}\,|\,a^{2m}+a^m-1=0\right\}$.
Consider the quasi-smooth hypersurface on $\mathbb{P}_{\Delta(0,2,1)}$ invariant by $\mathcal{F}$ given by
$$V=\left\{z_1^m+z_2^m+z_3^m=0\right\}.$$
It is easy seen that $\deg(\mathcal{F}) = (m-1,\left[0\right])$ and $\deg(V)=(m,\left[0\right])$, and so 
$\deg(V)_1=m=\deg(\mathcal{F})_1+1<\deg(\mathcal{F})_1+2$.
On the other hand, according to Theorem \ref{Mig1}, we have $P_{1,2}=-\frac{1}{m}z_2$, $P_{1,3}=0$ and $P_{2,3}=-\frac{1}{m}z_1$. 
Hence
$$\deg(P_{1,2})=(1,\left[2\right])=(m-1,\left[0\right])+(1,\left[0\right])+(1,\left[2\right])-(m,\left[0\right])=\deg(\mathcal{F})+\deg(z_1)+\deg(z_2)-\deg(V),$$
and
$$\deg(P_{2,3})=(1,\left[0\right])=(m-1,\left[0\right])+(1,\left[2\right])+(1,\left[1\right])-(m,\left[0\right])=\deg(\mathcal{F})+\deg(z_2)+\deg(z_3)-\deg(V).$$
Compare with Remark \ref{Mig3}-(ii).
\end{exe}

The next example is an application of Theorems \ref{Mig4} and \ref{Mig5}

\begin{exe}\label{A}
In homogeneous coordinates $z_{1,0}, z_{1,1}, z_{2,0}, z_{2,1}$, let us consider the multiprojective space 
$$\mathbb{P}^1\times\mathbb{P}^1=\left(\mathbb{C}^{2}\times\mathbb{C}^{2}-\mathcal{Z}\right)\,/\,(\mathbb{C}^{*})^2,$$
where $\mathcal{Z} = \left(\left\{0\right\}\times\mathbb{C}^{2}\right)\cup\left(\mathbb{C}^{2}\times\left\{0\right\}\right)$.
Let $\alpha_1$, $\alpha_2$, $\alpha$ be positive integers with $\alpha=\alpha_1+\alpha_2$ and $c_1,c_2\in\mathbb{C}^{\ast}$. 
Let $\mathcal{F}$ be a holomorphic foliation
on $\mathbb{P}^1\times\mathbb{P}^1$ induced by the quasi-homogeneous vector field $X=X_1+X_2$, where
$$X_1=c_1z_{1,0}^2z_{2,1}^{\alpha}\frac{\partial}{\partial z_{1,0}}-c_1z_{1,0}^2z_{2,0}^{\alpha_1}
(z_{2,0}^{\alpha_2}+z_{2,1}^{\alpha_2})\frac{\partial}{\partial z_{1,1}}$$
and
$$X_2=c_2z_{2,0}^2z_{2,1}^{\alpha_2-1}(\alpha_2z_{1,0}z_{2,0}^{\alpha_1}+\alpha z_{1,1}z_{2,1}^{\alpha_1})\frac{\partial}{\partial z_{2,0}}
-c_2z_{2,0}^{\alpha_1+1}(\alpha z_{1,0}z_{2,0}^{\alpha_2}+\alpha_1z_{1,0}z_{2,1}^{\alpha_2})\frac{\partial}{\partial z_{2,1}}.$$
Consider the quasi-smooth hypersurface on $\mathbb{P}^1\times\mathbb{P}^1$ invariant by $\mathcal{F}$ given by 
$$V=\left\{z_{1,0}z_{2,0}^{\alpha}+z_{1,1}z_{2,1}^{\alpha}+z_{1,0}z_{2,0}^{\alpha_1}z_{2,1}^{\alpha_2}=0\right\}.$$
Here, $V$ is invariant by $X_1$, $X_2$ and 
$\left\{0\right\} \subsetneq \mathrm{Sing} (V) = \mathbb{C}^{2}\times\left\{0\right\} \subset\mathcal{Z}$. 
Then, we have $\deg(V)=\deg(\mathcal{F})=(1,\alpha)$, and so 
$$\deg(V)_k \leq \deg(\mathcal{F})_k + 2,\,\,k=1, 2.$$ 
\end{exe}

Finally, the next example tells us that our hypotheses on $V$ are necessary are necessary

\begin{exe}\label{B}
In homogeneous coordinates $z_{1,0}, z_{1,1}, z_{2,0}, z_{2,1}$, let us consider the multiprojective space 
$$\mathbb{P}^1\times\mathbb{P}^1=\left(\mathbb{C}^{2}\times\mathbb{C}^{2}-\mathcal{Z}\right)\,/\,(\mathbb{C}^{*})^2,$$
where $\mathcal{Z} = \left(\left\{0\right\}\times\mathbb{C}^{2}\right)\cup\left(\mathbb{C}^{2}\times\left\{0\right\}\right)$.
Let $\mathcal{F}_1$ be a holomorphic foliation on $\mathbb{P}^1\times\mathbb{P}^1$ 
induced by the quasi-homogeneous vector field
$$X=z_{2,1}z_{1,0}^2\frac{\partial}{\partial z_{1,0}}+z_{1,1}z_{2,0}^2\frac{\partial}{\partial z_{2,0}}.$$
Consider the hypersurface on $\mathbb{P}^1\times\mathbb{P}^1$ invariant by $\mathcal{F}_1$ given by 
$$V=\left\{z_{1,0}^{n}z_{2,0}^{m}=0\right\}.$$
If $n,m>1$, then $\mathcal{Z} \subsetneq \mathrm{Sing}(V)
=\left(\left\{0\right\}\times\mathbb{C}\times\mathbb{C}\times\mathbb{C}\right)\cup
\left(\mathbb{C}\times\mathbb{C}\times\left\{0\right\}\times\mathbb{C}\right)$. We have $\deg(V)=(n,m)$ and  
$\deg(\mathcal{F}_1)=(1,1)$. Depending on $n$ and $m$, it is easy to see that the Poincar\'e problem could be or not could be valid. 

Similarly, let $\mathcal{F}_2$ be a holomorphic foliation on $\mathbb{P}^1\times\mathbb{P}^1$ induced by the quasi-homogeneous vector field
$$X=m z_{1,0}\frac{\partial}{\partial z_{1,0}}+n z_{2,0}\frac{\partial}{\partial z_{2,0}},$$
and consider the hypersurface on $\mathbb{P}^1\times\mathbb{P}^1$ 
invariant by $\mathcal{F}_2$ given by 
$$V=\left\{z_{1,0}^{n}z_{2,1}^{m}+z_{1,1}^{n}z_{2,0}^{m}=0\right\}.$$
Then $\deg(V)=(n,m)$ and $\deg(\mathcal{F}_2)=(0,0)$, and therefore we obtain the same conclusions. 
\end{exe}

\bigskip

\noindent{\footnotesize \textsc{Acknowlegments.} I am grateful to Mauricio Corr\^ea for helpful comments and suggestions during the preparation of this paper and for interesting conversations about toric varieties. 
I also thank Arturo Fern\'andez P\'erez and the referee for useful comments and suggestions which improved the readability of this paper.}


\end{document}